\documentclass[11pt]{amsart}
\usepackage[T1]{fontenc}
\usepackage{lmodern,amsmath,amssymb,amsthm,mathtools,mathrsfs,microtype, xcolor,xurl}
\usepackage[margin=1.2in]{geometry}
\usepackage{needspace}
\usepackage[normalem]{ulem}
\colorlet{BLUE}{blue}

\usepackage[colorlinks=true,linkcolor=blue,citecolor=blue,urlcolor=blue]{hyperref}
\numberwithin{equation}{section}
\newtheorem{theorem}{Theorem}[section]
\newtheorem{proposition}[theorem]{Proposition}
\newtheorem{lemma}[theorem]{Lemma}
\newtheorem{conjecture}[theorem]{Conjecture}
\newtheorem{remark}[theorem]{Remark}
\newcommand{\D}{\mathbb D}
\newcommand{\T}{\mathbb T}
\newcommand{\HH}{\mathscr H}

\newcommand{\OO}{\mathcal O}
\newcommand{\TD}{\mathcal T}
\newcommand{\KK}{\mathcal K}
\newcommand{\DD}{\mathcal D}
\newcommand{\tr}{\operatorname{tr}}
\newcommand{\ran}{\operatorname{ran}}
\newcommand{\Dom}{\operatorname{Dom}}
\newcommand{\rank}{\operatorname{rank}}

\title[Strict Dirichlet P\'olya Inequality]{Strict Dirichlet P\'olya Inequality for Simply Connected Planar Domains}
\date{\today}

\author{Quanyu Tang}
\address{School of Mathematical Sciences\\
University of Science and Technology of China\\
Hefei 230026, P.R. China}
\email{tangquanyu@mail.ustc.edu.cn}

\author{Zuoqin Wang}
\thanks{Z. Wang is partially supported by the National Natural Science Foundation of China (Grant No.~12571064).}
\address{School of Mathematical Sciences\\
University of Science and Technology of China\\
Hefei 230026, P.R. China}
\email{wangzuoq@ustc.edu.cn}

\begin{document}

\begin{abstract}
We establish the strict Dirichlet P\'olya inequality on every bounded simply connected planar domain. The proof uses a unitary boundary transport formulation of the Helmholtz equation and relates Dirichlet resonances to unitary eigenphase crossings through the Dirichlet-to-Neumann map. An exact area identity yields the sharp Weyl coefficient, while a Ky Fan estimate leaves a strictly positive remainder in the counting estimate. The general case follows by conformal exhaustion.
\end{abstract}

\maketitle

\section{Introduction}\label{d:sec:intro}

Let $\Omega\subset\mathbb R^2$ be a bounded planar domain.  The Dirichlet Laplacian is the self-adjoint operator associated with the closed quadratic form
\[
q_\Omega[u]=\int_\Omega |\nabla u|^2,\qquad u\in H_0^1(\Omega)
\]
on
$H_0^1(\Omega)$.  Equivalently, it is the Friedrichs extension of $-\Delta$ initially defined on $C_c^\infty(\Omega)$.
Its eigenvalues, repeated with multiplicity, are denoted by
\[
0<\lambda_1(\Omega)\le\lambda_2(\Omega)\le\cdots.
\] 
For comparison, whenever the Neumann Laplacian has discrete spectrum---in particular, for Lipschitz $\Omega$---we denote its eigenvalues by
\[
 0=\mu_0(\Omega)<\mu_1(\Omega)\le\mu_2(\Omega)\le\cdots.
\]
These standard form constructions and the min--max principle are used
throughout.

We use the counting conventions
\[
 N_D^{\le}(E)=\#\{j\ge1:\lambda_j(\Omega)\le E\},\qquad
 N_N^{<}(E)=\#\{j\ge0:\mu_j(\Omega)<E\}.
\]
When the domain must be displayed explicitly, we write
\[
    N_{D,\Omega}^{\le}(E)
    =\#\{j\ge1:\lambda_j(\Omega)\le E\}.
\]
With these opposite endpoint conventions, the strict
pointwise inequalities below are exactly equivalent to the strict
counting inequalities at every positive energy.

Under the standard hypotheses guaranteeing Weyl asymptotics, Weyl's law, originating in Weyl's pioneering work \cite{Weyl1911}, gives the common leading term for both counting functions (independently of the endpoint convention):
\[
 N_D^{\le}(E),\ N_N^{<}(E)\sim\frac{|\Omega|}{4\pi}E,
 \qquad E\to\infty.
\]
For piecewise smooth boundaries, Weyl further conjectured the two-term expansion \cite{Weyl1913}
\[
N_D^{\le}(E)
=\frac{|\Omega|}{4\pi}E
-\frac{|\partial\Omega|}{4\pi}E^{1/2}
+o(E^{1/2}),
\qquad
N_N^{<}(E)
=\frac{|\Omega|}{4\pi}E
+\frac{|\partial\Omega|}{4\pi}E^{1/2}
+o(E^{1/2}),
\]
which predicts that the first correction to the volume term is governed by the boundary length, with opposite signs for the Dirichlet and Neumann problems.

P\'olya first stated his conjecture in 1954 as an exercise entitled ``A New Conjecture'' in \emph{Mathematics and Plausible Reasoning} \cite{Polya54}. A commonly used modern formulation is the following pair of non-strict inequalities.
\begin{conjecture}[P\'olya's conjecture for planar domains]
    For every bounded planar domain $\Omega$ (assuming suitable boundary
    regularity for the Neumann case), and every $j\ge1$,\footnote{With our
        indexing $0=\mu_0<\mu_1\le\mu_2\le\cdots$, P\'olya's original
        formulation reads
        \[
        \mu_{j-1}(\Omega)<\frac{4\pi j}{|\Omega|}
    <\lambda_j(\Omega).
        \]
        Thus the original inequalities are strict, but the Neumann statement
        involves $\mu_{j-1}$ rather than $\mu_j$ and is therefore weaker by one
        eigenvalue index.  The formulation \eqref{eq:polya-conjecture} is the
        index-aligned version commonly used in the modern literature.}
    \begin{equation}\label{eq:polya-conjecture}
        \mu_j(\Omega)\le \frac{4\pi j}{|\Omega|}
        \le \lambda_j(\Omega).
    \end{equation}
\end{conjecture}

With the endpoint conventions fixed above, the counting-function form is
\[
N_D^{\le}(E)\le \frac{|\Omega|}{4\pi}E
\le N_N^{<}(E),
\qquad E>0.
\]

The result proved in this paper is the strict Dirichlet half of the modern
conjecture: for every bounded simply connected planar domain,
\[
 \lambda_j(\Omega)>\frac{4\pi j}{|\Omega|},\qquad j\ge1.
\]
Equivalently, $N_D^{\le}(E)<|\Omega|E/(4\pi)$ for every $E>0$.
Thus the Dirichlet inequality is strict, exactly as in P\'olya's original
formulation, but is proved here for the full class of bounded simply
connected planar domains.  The analogous non-strict P\'olya inequalities
may of course be posed in every Euclidean dimension.

P\'olya proved in 1961 the
corresponding non-strict inequalities for tiling domains \cite{Polya61}; the
additional restriction in his Neumann tiling argument was later
removed by Kellner \cite{Kellner}.

Recent progress has established the conjecture for a number of
non-tiling geometries. Filonov, Levitin, Polterovich, and Sher proved
the disk and sector cases, and the Dirichlet ball case in arbitrary
dimension \cite{FLPS}; their later work treats the Dirichlet problem for planar annuli
\cite{FLPSannuli}. The
remaining higher-dimensional Neumann ball problem was recently settled
independently in
\cite{FLPS26} and \cite{Li26}.

Other progress concerns domains with additional geometric structure.  Laptev proved that the P\'olya inequality passes from one factor to a Cartesian product \cite[Theorem~2.8]{Laptev}. More recently, He and one of the authors \cite{HW} established the conjecture for sufficiently thin product domains without assuming one factor to satisfy the P\'olya  inequality.
Jiang and Lin constructed in \cite{JiangLin} a further class obtained by removing pieces from domains for which
a sufficiently strong P\'olya estimate is already available, for example, strip-tiling domains.  For further results and related spectral bounds, see, for example, \cite{Berezin,Filonov,FrankLarson,FS,GMWZ,LiYau,Lin2017,Kroger}.

The purpose of this paper is to prove the strict Dirichlet P\'olya
inequality for every bounded simply connected planar domain.

\begin{theorem}
    \label{d:thm:main}
Let $\Omega\subset\mathbb R^2$ be a bounded nonempty simply connected
domain. Then, for every $j\ge1$,
\begin{equation}\label{d:eq:strict-eigenvalues}
 |\Omega|\lambda_j(\Omega)>4\pi j.
\end{equation}
Equivalently,
\begin{equation}\label{d:eq:polya}
 N_D^{\le}(E)<\frac{|\Omega|E}{4\pi}\qquad(E>0).
\end{equation}
\end{theorem}

\begin{remark}\label{rmk:neumann-strict-polya}
The strict Neumann inequality will be treated in a companion paper currently in preparation \cite{TW}. Together with the present theorem, the companion result will give the two-sided strict counting inequality
\begin{equation*}
    N_D^{\le}(E)<\frac{|\Omega|E}{4\pi}<N_N^{<}(E)
\end{equation*}
for any bounded simply connected Lipschitz planar domain $\Omega$.
\end{remark} 

\begin{remark}
	By \cite[Theorem 1.4]{SafarovFilonov2010}, together with the estimate obtained in the proof of their Theorem 2.1, if $\Omega$ is a bounded smooth planar domain, then
	\[
	N_D^{\le}(E)
	\le
	\frac{|\Omega|}{4\pi}E
	-\frac{|\partial\Omega|}{8\pi}E^{1/2}
	+o(E^{1/2}).
	\]
	Hence, for all sufficiently large $E$,
	\[
	N_D^{\le}(E)
	\le
	\frac{|\Omega|}{4\pi}E
	-\frac{|\partial\Omega|}{10\pi}E^{1/2}.
	\]
	Since only finitely many eigenvalues lie below the high-energy threshold, the strict inequalities \eqref{d:eq:strict-eigenvalues} at those eigenvalues provide a positive minimum gap. Decreasing the constant if necessary, we therefore obtain a uniform estimate
	\[
	N_D^{\le}(E)
	\le
	\frac{|\Omega|}{4\pi}\bigl(E-CE^{1/2}\bigr)_+,
	\qquad E>0,
	\]
	for every bounded simply connected smooth planar domain $\Omega$, where $C=C(\Omega)>0$.
	Thus, by adapting the removal arguments developed in \cite{JiangLin}; see also \cite{HeJiangLinWang,HW}, one can show that, after removing sufficiently small holes from $\Omega$, the resulting domain still satisfies the strict Dirichlet P\'olya inequality. In this way, one can construct many non-simply connected domains satisfying P\'olya's conjecture.
\end{remark}  
 
We now outline the proof and its main new ingredients.  We identify $\mathbb R^2$ with $\mathbb C$ and write $\D=\{z\in\mathbb C:|z|<1\}$, with $\T=\partial\D$. Fix an energy
$E=k^2>0$.  We first assume that $\Omega=F(\D)$, where $F$ is holomorphic
and injective on a neighborhood of $\overline\D$, and consider the nested
domains $\Omega_r=F(r\D)$.  We call $r$ a Dirichlet resonance
at energy $E$ if $E$ is a Dirichlet eigenvalue of $\Omega_r$;
otherwise, $r$ is nonresonant.  By continuity and strict domain monotonicity,
every ordered Dirichlet eigenvalue that reaches $E$ does so at a unique
radius.  Our aim is to count these resonances with multiplicity by relating
them to the motion of a unitary family whose total variation is controlled
by the area of $\Omega$.

Section~\ref{d:sec:transport} starts from the boundary derivative sequence
$\partial u,\partial^2u,\ldots$ of a Helmholtz solution. The Helmholtz equation and the boundary
chain rule yield a coupled first-order differential
system for the boundary derivatives
$\partial u,\partial^2u,\ldots$.
After normalization, its coefficients can be expressed
in terms of the unilateral shift $S$ and its adjoint. Its homogeneous
Hilbert-space realization defines a unitary boundary transport and the
monodromy $V_r$.  An important difficulty is that the derivative sequence
of an actual Dirichlet eigenfunction need not belong to $\ell^2$.
We therefore construct the transport using Herglotz waves, whose
normalized derivatives are square-summable, and recover the Dirichlet
spectral information later from boundary data.   

In Section~\ref{d:sec:geometry}, we vary the domain radially and remove
the effect of the moving starting point of its boundary by unitary
conjugation.  We then pass to the adjoint unitary family $U_r$.  The
rank-one identity $[S^*,S]=P_0$, where $P_0$ projects onto the
first coordinate of $\ell^2$, yields a radial generator $L_{r,k}=-iU_r^*\partial_rU_r$ that is positive, injective,
and trace class, together with the exact area identity
\[
  \int_0^1\operatorname{tr}L_{r,k}\,dr
  =\frac{k^2}{2}|\Omega|.
\]
Positivity determines the direction of the unitary eigenphases, while
the trace identity gives precisely the coefficient required in P\'olya's
inequality.  This is the first central new ingredient of the proof.

Sections~\ref{d:sec:dtn-pole} and~\ref{d:sec:dirichlet-branch} link this
unitary motion to the actual Dirichlet spectrum. Under the conformal change of variables $z\mapsto F(rz)$, the Helmholtz equation on $\Omega_r$ becomes $(-\Delta-k^2w_r)\widetilde u=0$ on the fixed unit disk $\D$, where $w_r(z)=r^2|F'(rz)|^2$. Let $A_r=-\Delta-k^2w_r$ be the corresponding operator on $L^2(\D)$ with Dirichlet boundary condition. A resonance of multiplicity $m$ at $r=r_*$ is then equivalent to a zero eigenvalue of $A_{r_*}$ of multiplicity $m$. As $r$ approaches $r_*$ from below, the corresponding $m$ eigenvalues of $A_r$ approach zero from the positive side. 
Consequently, the Dirichlet-to-Neumann quadratic form diverges to $-\infty$ on $m$ independent boundary directions.  Away from resonance, Herglotz traces are dense and the
monodromy $V_r$ has no fixed vector.  
A boundary identity then relates the Dirichlet-to-Neumann map to the Cayley transform of $V_r$, with a remainder bounded on the relevant boundary data. Through this identity, the negative divergence of the Dirichlet-to-Neumann quadratic form implies that at least $m$ eigenphases of $U_r$ approach $2\pi$ from below. This multiplicity-preserving correspondence is the
second central ingredient.

Finally, Section~\ref{d:sec:crossing-count} uses positivity of the
radial generator to count the resulting phase crossings.  The Ky Fan
principle bounds the motion of any $m$ selected phase directions by
$\Phi_m(L_{r,k})$, the sum of the $m$ largest eigenvalues of the
generator.  Counting all Dirichlet resonances, including endpoint
multiplicities, and letting the reference phase approach $2\pi$ gives
\[
  2\pi m\le\int_0^1\Phi_m(L_{r,k})\,dr.
\]
The area identity involves the full trace, whereas the counting
argument uses only this partial sum.  Because $L_{r,k}$ is injective
on the infinite-dimensional space $\ell^2$, the remaining trace is
strictly positive and gives the strict inequality.  The conformal exhaustion argument preserves the strictly positive remainder and extends the conclusion to every bounded simply connected planar domain, completing the proof of Theorem~\ref{d:thm:main} in Subsection~\ref{d:sec:completion}.

\subsection*{Lean formalizations}

Both Theorem~\ref{d:thm:main} and the strict Neumann
P\'olya inequality for bounded simply connected
Lipschitz planar domains, discussed in
Remark~\ref{rmk:neumann-strict-polya},
have been formalized in Lean~4. 
Their formalizations have been verified and registered
in the Palomar Registry as
\href{https://palomar-registry.org/entry?id=PALOMAR-2026-10-09-000003&version=1}
{entry PALOMAR-2026-10-09-000003}
and
\href{https://palomar-registry.org/entry?id=PALOMAR-2026-10-09-000007&version=1}
{entry PALOMAR-2026-10-09-000007},
respectively. 
The corresponding machine-checked developments,
together with their source code and reproducibility
instructions, are publicly available in the
\href{https://github.com/QuanyuTang/polya-dirichlet-simply-connected-formalization-Palomar-version}
{Dirichlet formalization repository} and the
\href{https://github.com/QuanyuTang/polya-neumann-simply-connected-formalization-Palomar-version-v2}
{Neumann formalization repository}, respectively.

\subsection*{Statement on AI usage}  
The authors acknowledge the use of AI tools for exploratory discussions and language refinement during the preparation of this work. 
ChatGPT assisted in developing certain technical
arguments in the proof of the main theorem, notably
the finite-rank approximation and the passage to
the infinite-dimensional spectral-flow setting
in the proof of Lemma~\ref{d:lem:phase}.  
Harmonic's Aristotle and OpenAI's Codex assisted with
the Lean~4 formalizations of the main theorem of this
paper. 
The authors have reviewed and revised all AI-assisted material and take full responsibility for all mathematical statements, proofs, and conclusions in this paper. 

\section{Boundary jets and unitary transport} 
\label{d:sec:transport}

\subsection{Boundary jets of Dirichlet eigenfunctions}

For now let $\Omega$ be a smooth Jordan domain and let
$\gamma:[0,2\pi]\to\partial\Omega$ be a positively oriented smooth
regular parametrization. Suppose that $u$ is a Dirichlet eigenfunction
of $\Omega$ at energy $E=k^2>0$:
\begin{equation}\label{d:eq:starting-eigenfunction}
    \left\{
    \begin{array}{ll}
        -\Delta u=k^2u, & \text{in }\Omega,\\
        u\circ\gamma=0. &
    \end{array}
    \right.
\end{equation}
Smooth-boundary elliptic regularity gives
$u\in C^\infty(\overline\Omega)$; see \cite{LM,McLean}.

The purpose of this subsection is to convert the Helmholtz equation
and the Dirichlet boundary condition into a closed first-order system
for a one-sided sequence of derivatives along $\partial\Omega$.
Identify the plane with $\mathbb C$ and write
\[
\partial=\tfrac12(\partial_x-i\partial_y),
\qquad
\bar\partial=\tfrac12(\partial_x+i\partial_y).
\]
Since $\Delta=4\partial\bar\partial$, the equation
\eqref{d:eq:starting-eigenfunction} gives
\begin{equation}\label{d:eq:raw-recursion}
    \bar\partial\,\partial^{n+1}u
    =-\frac{k^2}{4}\partial^n u,
    \qquad n\ge0.
\end{equation}

Introduce the normalized derivatives
\begin{equation*}
    f_n(z)
    =
    \left(\frac{2i}{k}\right)^{n+1}
    \partial^{\,n+1}u(z),
    \qquad n\ge0.
\end{equation*}
The normalization is chosen so that the two neighboring terms in the
boundary recursion have the same coefficient. Indeed,
\begin{equation}\label{d:eq:jet-recursion-dirichlet}
    \partial f_n=-\frac{ik}{2}f_{n+1},
    \qquad n\ge0,
\end{equation}
while \eqref{d:eq:raw-recursion} gives
\[
\bar\partial f_n
=
-\frac{ik}{2}f_{n-1},
\qquad n\ge1,
\]
and, for the first component,
\[
\bar\partial f_0=-\frac{ik}{2}u.
\]

We now restrict \eqref{d:eq:jet-recursion-dirichlet} and the two companion
identities to the boundary and, for simplicity, continue to write
\[
f_n(\theta):=f_n(\gamma(\theta)).
\]
We thus get a one-sided sequence of functions
\[
f=(f_0,f_1,f_2,\ldots)
\]
on the boundary.
Moreover, by the chain rule,
\[
f_n'
=
\gamma'\,\partial f_n
+
\overline{\gamma'}\,\bar\partial f_n.
\]
Since $u\circ\gamma=0$, one gets
\begin{align*}
    f_0'
    &=
    -\frac{ik}{2}\gamma' f_1,\\
    f_n'
    &=
    -\frac{ik}{2}
    \left(
    \gamma'f_{n+1}
    +
    \overline{\gamma'}f_{n-1}
    \right),
    \qquad n\ge1.
\end{align*}

\subsection{Herglotz waves}

In general, the full derivative sequence $f$ need not take values in $\ell^2$. We therefore realize the Hilbert-space transport first on Herglotz waves, whose normalized derivative sequences are square-summable. The approximation needed to connect this transport with Dirichlet boundary spectral data is stated in Lemma~\ref{d:lem:density}. Thus consider Herglotz waves, i.e., functions of the form
\begin{equation}\label{d:eq:herglotz}
    u_a(z)=\int_0^{2\pi}a(\phi)
    \exp\left[-\frac{ik}{2}(ze^{-i\phi}+\bar z e^{i\phi})\right]
    \frac{d\phi}{2\pi},
\end{equation}
where $a\in L^2(\T,d\phi/(2\pi))$.  Each integrand solves $(\Delta+k^2)u=0$, hence so does $u_a$. The density property required below is a boundary-trace statement off the Dirichlet spectrum; it is proved in Lemma~\ref{d:lem:density} using the standard Herglotz approximation theorem \cite{CK}.

To distinguish this sequence from the componentwise sequence $f$ attached to the actual eigenfunction, denote the normalized derivatives of $u_a$ by
\begin{equation*}
    y_n(z)
    =
    \left(\frac{2i}{k}\right)^{n+1}
    \partial^{n+1}u_a(z),
    \qquad n\ge0.
\end{equation*}

For a Herglotz wave $u_a$, by contrast, differentiating the integrand in \eqref{d:eq:herglotz} once with
respect to $z$ multiplies it by $-ik e^{-i\phi}/2$. Hence
\[
y_n(z)
=
\int_0^{2\pi}
a(\phi)e^{-i(n+1)\phi}
e^{-ik\operatorname{Re}(ze^{-i\phi})}
\frac{d\phi}{2\pi}.
\]
Thus $(y_n(z))_{n\ge0}$ consists of the positive-index Fourier
coefficients of the $L^2$ function
\[
\phi\longmapsto
a(\phi)e^{-ik\operatorname{Re}(ze^{-i\phi})}.
\]
By Parseval,
\begin{equation}\label{d:eq:jet-parseval}
    \sum_{n\ge0}|y_n(z)|^2
    \le \|a\|_{L^2}^2.
\end{equation}
In particular,
\[
y(z)=(y_0(z),y_1(z),\ldots)\in
\HH:=\ell^2(\mathbb N_0).
\]
Throughout the paper, complex inner products are conjugate-linear in the first variable and linear in the second.

Along the boundary, the sequence $y$ satisfies the same recurrence
as the derivative sequence in the preceding subsection, except that
the Herglotz wave $u_a$ itself need not vanish on $\partial\Omega$. Writing
\[
h=u_a\circ\gamma,
\]
we thus obtain
\begin{align*}
    y_0'
    &=
    -\frac{ik}{2}\gamma' y_1
    -\frac{ik}{2}\overline{\gamma'}\,h,
    \\
    y_n'
    &=
    -\frac{ik}{2}
    \left(
    \gamma'y_{n+1}
    +
    \overline{\gamma'}y_{n-1}
    \right),
    \qquad n\ge1.
\end{align*}

Let $(e_n)_{n\ge0}$ be the standard basis of
$\HH=\ell^2(\mathbb N_0)$ and let $S$ be the unilateral shift,
\[
Se_n=e_{n+1}.
\]
For a generic vector $\xi=(\xi_n)_{n\ge0}\in\HH$, one has $(S^*\xi)_n=\xi_{n+1}$, whereas $(S\xi)_0=0$ and $(S\xi)_n=\xi_{n-1}$ for $n\ge1$.
So the preceding equations take the form
\begin{equation}\label{d:eq:driven-jet}
    y'
    =
    C_\theta y
    -
    \frac{ik}{2}\overline{\gamma'}\,h\,e_0,
    \qquad
    C_\theta
    =
    -\frac{ik}{2}
    \left(
    \gamma'S^*
    +
    \overline{\gamma'}S
    \right).
\end{equation}
Thus $C_\theta$ is the $\ell^2$ realization of the coefficient pattern
appearing in the componentwise Dirichlet boundary recursion.

The homogeneous equation
\begin{equation}\label{d:eq:homogeneous-transport}
    y'=C_\theta y
\end{equation}
will be used to define the unitary transport and its monodromy.

The connection with the actual Dirichlet eigenspace will be made later through the negative divergence of the Dirichlet-to-Neumann quadratic form, rather than by assigning an $\ell^2$ derivative sequence to each eigenfunction.

The following lemma will be used later:
\begin{lemma}
    \label{d:lem:zeroth-observation}
    Let $\gamma$ be a regular boundary parametrization and let
    $y:[0,2\pi]\to\HH$ satisfy
    \[
    y'
    =
    -\frac{ik}{2}
    \bigl(\gamma'S^*+\overline{\gamma'}S\bigr)y.
    \]
    If $
    y_0(\theta)=0$ for all $\theta$, then $
    y(\theta)=0$ for all $\theta$.
\end{lemma}

\begin{proof}
    Since $y_0=0$, the equation for the zeroth component gives
    \[
    0=y_0'
    =-\frac{ik}{2}\gamma' y_1.
    \]
    Because $\gamma'\ne0$, we obtain $y_1=0$.

    Suppose inductively that $y_0=\cdots=y_n=0$.
    Then the equation for the $n$th component reduces to
    \[
    0=y_n'
    =
    -\frac{ik}{2}\gamma' y_{n+1},
    \]
    and hence $y_{n+1}=0$.  Thus every component of $y$ vanishes.
\end{proof}

\subsection{Unitary boundary transport and monodromy}

We now consider the fundamental solution of the homogeneous Hilbert-space equation \eqref{d:eq:homogeneous-transport}:
\begin{equation}\label{d:eq:transport}
 \partial_\theta W_{k,\gamma}
    =
    C_\theta W_{k,\gamma},
    \qquad
    W_{k,\gamma}(0)=I.
\end{equation}
Since $C_\theta$ is bounded and depends continuously on $\theta$,
the operator-valued initial value problem \eqref{d:eq:transport}
has a unique solution on $\HH=\ell^2(\mathbb N_0)$.
For every $v\in\HH$, the function
\[
y(\theta)=W_{k,\gamma}(\theta)v
\]
is the unique solution of
\[
y'=C_\theta y,\qquad y(0)=v.
\]
Thus $W_{k,\gamma}(\theta)$ transports the initial datum at $\theta=0$ to the
point $\theta$ along the boundary.

Moreover, since $C_\theta^*=-C_\theta$, writing $W(\theta)=W_{k,\gamma}(\theta)$, we have
\[
\frac{d}{d\theta}\bigl(W(\theta)^*W(\theta)\bigr)
=
W(\theta)^*(C_\theta^*+C_\theta)W(\theta)
=0.
\]
Hence
\[
W(\theta)^*W(\theta)=I.
\]
The evolution operator of this bounded operator-valued differential equation is invertible: its inverse is obtained by solving the backward equation. Since $W(\theta)^*$ is already a left inverse, it follows that $W(\theta)^{-1}=W(\theta)^*$, and hence $W(\theta)$ is unitary for every $\theta$.
We therefore call $W$ the unitary boundary transport.
Its value after one circuit,
\[
V_{k,\gamma}=W_{k,\gamma}(2\pi),
\]
is the corresponding monodromy operator.  When no confusion can arise, we abbreviate $W_{k,\gamma}$ and $V_{k,\gamma}$ by $W$ and $V$.

For a Herglotz wave with boundary trace $h=u_a\circ\gamma$ and initial
derivative vector $v=y(0)\in\HH$, variation of constants gives
\begin{equation}\label{d:eq:variation-first}
y(\theta)
=
W_{k,\gamma}(\theta)v
-
\frac{ik}{2}W_{k,\gamma}(\theta)
\int_0^\theta
W_{k,\gamma}(s)^*e_0
\overline{\gamma'(s)}\,h(s)\,ds.
\end{equation}
Since its normalized derivative sequence is
single-valued along the closed boundary, we have
\[
y(2\pi)=y(0)=v.
\]
Evaluating the variation-of-constants formula at $\theta=2\pi$
therefore gives
\begin{equation}\label{d:eq:periodicity-first}
    (I-V_{k,\gamma})v
    =
    -\frac{ik}{2}V_{k,\gamma}
    \int_0^{2\pi}
    W_{k,\gamma}(s)^*e_0\overline{\gamma'(s)}h(s)\,ds,
    \qquad h=u_a\circ\gamma .
\end{equation}
Thus the difference $(I-V)v$ is determined entirely by the boundary trace $h$. In particular, if the Herglotz boundary trace is small, then $(I-V)v$ is small.

Equation \eqref{d:eq:periodicity-first} will later be combined with the Dirichlet-to-Neumann map. Near a Dirichlet resonance, its quadratic form diverges to $-\infty$ on a finite-dimensional space of boundary data. Density of Herglotz traces then allows us to relate this divergence to the monodromy. This mechanism gives unitary eigenvalues approaching the reference point $1$. The direction of approach is stated only after the oppositely oriented radial family $U_r$, unitarily equivalent to $V_r^*$, has been introduced; it is the $U_r$-phases whose arguments approach $2\pi$ from below.  For the radial family introduced in the next section, Lemma~\ref{d:lem:injective} will show that $1$ is not an eigenvalue of the monodromy at a nonresonant radius. Thus the point $1$ on the unit circle serves as the reference for detecting a Dirichlet resonance. At a nonresonant radius, $1$ is not an eigenvalue of $V_r$, although it remains in its spectrum because $V_r-I$ is compact on the infinite-dimensional space $\HH$.

We remark that the boundary point $\gamma(0)=z_0$ is only the initial point for the boundary transport equation.  An orientation-preserving reparametrization that fixes the starting point leaves $V_{k,\gamma}$ unchanged, while changing the starting point conjugates the monodromy by a unitary transport. Its spectrum
therefore does not depend on that choice.

\section{The positive radial generator and the area identity}
\label{d:sec:geometry}

\subsection{The rank-one commutator of the unilateral shift}

We have constructed a unitary boundary transport and its monodromy.
We now vary the domain and study the resulting motion of the monodromy
spectrum on the unit circle.  The key algebraic input is the
one-dimensional defect of the unilateral shift. We write $[A,B]=AB-BA$ for the commutator. The trace-class ideal is denoted by $\mathcal S_1$, with norm $\|A\|_{\mathcal S_1}=\operatorname{tr}|A|$. If $P_0$ denotes the
orthogonal projection onto $\mathbb C e_0$, then
\begin{equation}\label{d:eq:commutator}
    S^*S=I,\qquad
    SS^*=I-P_0,\qquad
    [S^*,S]=P_0.
\end{equation}
The origin of this defect is already visible in the boundary
recursion: because of the Dirichlet condition, the equation for
$f_0$ contains no term involving a preceding component, whereas for
$n\ge1$ the equation for $f_n$ couples both $f_{n+1}$ and $f_{n-1}$.
This missing backward term at $n=0$ is exactly what turns the
bilateral coupling pattern into a unilateral shift and produces the
rank-one commutator in \eqref{d:eq:commutator}.
Its relevance to the transport equation \eqref{d:eq:transport} becomes visible by writing the two real 
coordinate generators. Put
\[
X=-\frac{ik}{2}(S+S^*),\qquad
Y=\frac{k}{2}(S^*-S).
\]
For $\gamma=x+iy$, the equation \eqref{d:eq:driven-jet} can be rewritten as $C_\theta=x_\theta X+y_\theta Y$, and
\[
[X,Y]=-\frac{ik^2}{4}[S+S^*,S^*-S]
=\frac{ik^2}{2}P_0.
\]
The positivity and trace of this rank-one commutator are precisely
what will produce, respectively, the positivity of the radial generator
and the area coefficient in its trace.

 \subsection{Radial domain variation at fixed energy}
 For this stage assume $\Omega=F(\D)$, where $F$ is holomorphic and
 injective on a neighborhood of $\overline\D$.  This temporary
 assumption allows us to differentiate a smooth nested family of
 domains.  It will later be removed by conformal exhaustion in Subsection~\ref{d:sec:completion}.  Fix the energy $k^2>0$ throughout
 this subsection and set
 \[
 \Omega_r=F(r\D),
 \qquad
 \gamma_r(\theta)=F(re^{i\theta}),
 \qquad
 0\le r\le1.
 \]
 Write
 \[
 \gamma_r(\theta)=x(r,\theta)+iy(r,\theta).
 \]

 For each fixed $r$, the transport constructed in the previous section
 applied to the closed curve
 $\theta\mapsto\gamma_r(\theta)$ gives a unitary operator
 $W_r(\theta)$ satisfying
 \[
 \partial_\theta W_r(\theta)
 =
 C_\theta(r,\theta)W_r(\theta),
 \qquad
 W_r(0)=I,
 \]
 where  the transport coefficient
 \[
 C_\theta(r,\theta)
 =
 \partial_\theta x(r,\theta)\,X
 +
 \partial_\theta y(r,\theta)\,Y.\]
 Its value after one circuit, $V_r$,  is the monodromy associated with $\partial\Omega_r$.

 There is one point that must be taken into account before the
 monodromies $V_r$ can be differentiated with respect to $r$.
 Their starting point is not fixed: it moves along the path
 \[
 z_r:=\gamma_r(0)=F(r).
 \]
 Recall that the transport coefficient along a planar path depends
 linearly on the velocity of that path.  We therefore use the same
 transport rule along $r\mapsto z_r$.  Define
 \[
 B_r
 :=
 \partial_r x(r,0)\,X
 +
 \partial_r y(r,0)\,Y,
 \]
 and let $R_r$ be the corresponding transport:
 \begin{equation}\label{d:eq:basepoint-transport}
     \partial_rR_r=B_rR_r,
     \qquad
     R_0=I.
 \end{equation}
 Since $B_r^*=-B_r$, each $R_r$ is unitary.

 The operator $R_r$ defined by \eqref{d:eq:basepoint-transport} allows us to compare all the monodromies relative
 to the fixed reference point $z_0=F(0)$.  Indeed, starting with
 $v\in\HH$, we may first transport it from $z_0$ to $z_r$, then go
 once around $\partial\Omega_r$, and finally transport it back:
 \[
 v
 \xmapsto{\ R_r\ }
 R_rv
 \xmapsto{\ V_r\ }
 V_rR_rv
 \xmapsto{\ R_r^*\ }
 R_r^*V_rR_rv.
 \]
 This motivates
 \begin{equation}\label{d:eq:radial-based}
     H_r:=R_r^*V_rR_r.
 \end{equation}
 Since $R_r$ is unitary, $H_r$ and $V_r$ have exactly the same
 spectrum.  The advantage of $H_r$ is that its radial derivative no
 longer contains the artificial contribution coming from the motion
 of the starting point.

We now compute the derivative with respect to $r$ of the operator
$H_r$ defined in \eqref{d:eq:radial-based}. To this end, introduce
the coefficient associated with the radial velocity of the
two-parameter family $\gamma_r(\theta)$:
 \[
 D_r(\theta)
 :=
 \partial_r x(r,\theta)\,X
 +
 \partial_r y(r,\theta)\,Y.
 \]
 Thus $D_r(0)=D_r(2\pi)=B_r$.

 \begin{lemma}
     \label{d:lem:radial-variation}
     For $r>0$,
     \begin{equation}\label{d:eq:radial-Q}
         H_r^*\partial_rH_r=-iQ_r,
     \end{equation}
     where
     \begin{equation}\label{d:eq:Qr}
         Q_r=
         \frac{k^2}{2}R_r^*
         \left(
         \int_0^{2\pi}
         J(r,\theta)\,
         W_r(\theta)^*P_0W_r(\theta)\,d\theta
         \right)R_r
     \end{equation}
     and
     \begin{equation}\label{d:eq:Jacobian}
         J(r,\theta)
         =
         \partial_r x\,\partial_\theta y
         -
         \partial_\theta x\,\partial_r y
         =
         r|F'(re^{i\theta})|^2.
     \end{equation}
     In particular, $Q_r$ is a positive trace-class operator.
 \end{lemma}

 \begin{proof}
     Since mixed derivatives commute,
     \[
     \partial_rC_\theta=\partial_\theta D_r.
     \]
     Differentiating the boundary transport equation $
     \partial_\theta W_r=C_\theta W_r$
     with respect to $r$ gives
     \begin{align*}
         \partial_\theta
         \left(
         W_r^*\partial_rW_r-W_r^*D_rW_r
         \right) =
         W_r^*
         \left(
         \partial_rC_\theta-\partial_\theta D_r
         +[C_\theta,D_r]
         \right)
         W_r  =
         W_r^*[C_\theta,D_r]W_r.
     \end{align*}
     Using $
     [X,Y]=\frac{ik^2}{2}P_0$
     we obtain
     \[
     [C_\theta,D_r]
     =
     -\frac{ik^2}{2}
     \left(
     \partial_r x\,\partial_\theta y
     -
     \partial_\theta x\,\partial_r y
     \right)P_0
     =
     -\frac{ik^2}{2}J(r,\theta)P_0.
     \]
     Hence
     \begin{equation}\label{d:eq:variation-W}
         \partial_\theta
         \left(
         W_r^*\partial_rW_r-W_r^*D_rW_r
         \right)
         =
         -\frac{ik^2}{2}
         J(r,\theta)W_r^*P_0W_r.
     \end{equation}

     At $\theta=0$, we have $W_r(0)=I$ and
     $\partial_rW_r(0)=0$, while $D_r(0)=B_r$.  At
     $\theta=2\pi$, we have $W_r(2\pi)=V_r$ and
     $D_r(2\pi)=B_r$.  Integrating
     \eqref{d:eq:variation-W} therefore gives
     \begin{equation}\label{d:eq:variation-V}
         V_r^*\partial_rV_r
         =
         V_r^*B_rV_r-B_r
         -
         \frac{ik^2}{2}
         \int_0^{2\pi}
         J(r,\theta)W_r(\theta)^*P_0W_r(\theta)\,d\theta.
     \end{equation}

     Now differentiate $H_r=R_r^*V_rR_r$. Since
     \[
     \partial_rR_r=B_rR_r,
     \qquad
     \partial_rR_r^*=-R_r^*B_r,
     \]
     we find
     \[
     H_r^*\partial_rH_r
     =
     R_r^*
     \left(
     V_r^*\partial_rV_r
     -
     V_r^*B_rV_r+B_r
     \right)
     R_r.
     \]
     The first two terms on the right-hand side of
     \eqref{d:eq:variation-V} are therefore cancelled exactly by the
     derivatives of the two conjugating factors.  What remains is
     \eqref{d:eq:radial-Q}--\eqref{d:eq:Qr}.

     Finally,
     \[
     J(r,\theta)=r|F'(re^{i\theta})|^2>0
     \qquad (r>0),
     \]
     and $P_0\ge0$.  Since unitary conjugation preserves positivity,
     \eqref{d:eq:Qr} shows that $Q_r\ge0$.  Its integrand has rank one,
     so $Q_r\in\mathcal S_1$.
 \end{proof}

 As a consequence, we see
 \[
 -iH_r^*\partial_rH_r=-Q_r\le0.
 \]
So the eigenphases of $H_r$ move in the negative direction as $r$
 increases.  For the counting argument it is more convenient to use
 the opposite orientation.  We therefore define
 \begin{equation}\label{d:eq:Ur-definition}
     U_r:=H_r^*=R_r^*V_r^*R_r.
 \end{equation}
 We also need the instantaneous radial generator of the unitary path
 $r\mapsto U_r$, defined by
 \begin{equation}\label{d:eq:Lr-definition}
     L_{r,k}:=-iU_r^*\partial_rU_r.
 \end{equation}
 It is self-adjoint, since $U_r$ is unitary, and equivalently
 \[
 \partial_rU_r=iU_rL_{r,k}.
 \]
 Thus $L_{r,k}$ from \eqref{d:eq:Lr-definition} governs the infinitesimal motion of the unitary
 spectrum as the domain expands.
 Equation \eqref{d:eq:radial-Q} gives
 \begin{equation}\label{d:eq:Lr-Qr}
     L_{r,k}=H_rQ_rH_r^*\ge0.
 \end{equation}
 Thus the phases of $U_r$ move in the positive direction as the
 domain expands.

 The operator $Q_r$ is trace class because its integrand in
 \eqref{d:eq:Qr} is a scalar multiple of the rank-one operator
 $W_r^*P_0W_r$.  The standard trace-ideal operations used below are
 described in \cite{Simon}.

 \begin{proposition}
     \label{d:prop:radial}
     For $r>0$, the operator $L_{r,k}$ is positive, injective, and trace
     class.  Moreover,
     \[
     U_r-I,\qquad V_r-I\in\mathcal S_1.
     \]
     For every $0<b\le1$,
     \begin{equation}\label{d:eq:radial-area}
         \int_0^b\tr L_{r,k}\,dr
         =
         \frac{k^2}{2}|F(b\D)|.
     \end{equation}
     The map $(r,k)\mapsto L_{r,k}$ is jointly continuous in trace norm,
     including at $r=0$, and $r\mapsto U_r$ is locally real analytic in
     operator norm for $r>0$.
 \end{proposition}

 \begin{proof}
By \eqref{d:eq:Lr-Qr}, $L_{r,k}$ is already known to be positive and
     trace class.  We now prove the remaining assertions of the proposition.

     We first prove injectivity.  Since $L_{r,k}$ is unitarily equivalent
     to $Q_r$, it is enough to show that $Q_r$ has trivial kernel.
     Suppose $\langle v,Q_rv\rangle=0$. By \eqref{d:eq:Qr},
     \[
     0
     =
     \frac{k^2}{2}
     \int_0^{2\pi}
     J(r,\theta)
     \left|
     \left\langle
     e_0,W_r(\theta)R_rv
     \right\rangle
     \right|^2d\theta.
     \]
     Because $J(r,\theta)>0$, continuity implies $
     \left\langle
     e_0,W_r(\theta)R_rv
     \right\rangle=0$
      for every $\theta$. Thus if we set $
     y(\theta):=W_r(\theta)R_rv$,
     then $y$ is an $\HH$-valued solution of the homogeneous transport
     equation $y'(\theta)=C_\theta y(\theta)$.
     Its zeroth component vanishes identically.   Lemma~\ref{d:lem:zeroth-observation} therefore
     gives $y\equiv0$.  Since $R_r$ is unitary, $v=0$.

     We next compute the trace.  By
     \eqref{d:eq:Lr-Qr}, unitary invariance of the trace, and
     $\tr(W_r^*P_0W_r)=\tr P_0=1$,
     \begin{align}
         \tr L_{r,k}
          =
         \tr Q_r  =
         \frac{k^2}{2}
         \int_0^{2\pi}
         J(r,\theta)\,d\theta  =
         \frac{k^2}{2}
         \int_0^{2\pi}
         r|F'(re^{i\theta})|^2\,d\theta.
         \label{d:eq:radial-trace}
     \end{align}
     By \eqref{d:eq:Jacobian}, the last integrand is the Jacobian of the
conformal map in polar coordinates. Therefore
     \[
     \int_0^b\!\int_0^{2\pi}
     r|F'(re^{i\theta})|^2\,d\theta\,dr
     =
     \int_{b\D}|F'(z)|^2\,dA(z)
     =
     |F(b\D)|.
     \]
     Integrating \eqref{d:eq:radial-trace} proves
     \eqref{d:eq:radial-area}.

For completeness, the continuity here is in the trace norm.  On compact parameter sets, the coefficients of the transport equations, and hence $W_r$ and $R_r$, depend uniformly continuously on $(r,k)$ in operator norm. More precisely, the trace-class-valued integrand in \eqref{d:eq:Qr}, after including the unitary conjugation by $R_r$, satisfies
\[
\left\|
\frac{k^2}{2}J(r,\theta)R_r^*W_r(\theta)^*P_0W_r(\theta)R_r
\right\|_{\mathcal S_1}
=
\frac{k^2}{2}J(r,\theta).
\]
On every compact $(r,k)$-parameter set these scalar trace norms admit an $L^1(d\theta)$ dominator. Hence the integral is a Bochner integral with values in $\mathcal S_1$, and dominated convergence gives joint trace-norm continuity of $Q_r$. Together with the operator-norm continuity of $H_r$, the identity $L_{r,k}=H_rQ_rH_r^*$ then gives joint trace-norm continuity of $L_{r,k}$. At $r=0$ this follows directly from
     the factor
     \[
     J(r,\theta)=r|F'(re^{i\theta})|^2.
     \]
     Since $\gamma_0$ is constant,
     \[
     V_0=H_0=U_0=I.
     \]
     From the definition of $L_{r,k}$, $\partial_rU_r=iU_rL_{r,k}$,
     and integration in trace norm therefore gives
     \[
     U_r-I\in\mathcal S_1.
     \]
     Consequently $H_r-I=U_r^*-I\in\mathcal S_1$.
     Since $V_r=R_rH_rR_r^*$,
     unitary conjugation gives
     \[
     V_r-I\in\mathcal S_1.
     \]

     Finally, because $F$ is holomorphic on a neighborhood of
     $\overline\D$, the coefficients of the transport equations depend
     real analytically on $r$ for $r>0$.  Their Picard series therefore
     give local operator-norm analyticity of $W_r$, $R_r$, and hence
     $U_r$.
 \end{proof}

\subsection{Strict positivity of the spectral tail}

The area identity \eqref{d:eq:radial-area} controls the full trace of the positive operator
$L_{r,k}$.  For the strict P\'olya inequality, however, it is important
not to use the whole trace in counting the first $m$ spectral
crossings.

For a positive trace-class operator $A$, let $\Phi_m(A)$
denote the sum of its $m$ largest eigenvalues, counted with
multiplicity, and define the remaining spectral tail by
\begin{equation}\label{d:eq:spectral-tail}
\tau_m(A):=\tr A-\Phi_m(A).
\end{equation}
Since $L_{r,k}$ is positive, compact, and injective on the
infinite-dimensional space $\HH$, it has infinitely many positive
eigenvalues.  Hence, for every $r>0$ and every finite $m$,
\begin{equation}\label{d:eq:positive-tail}
    \tau_m(L_{r,k})>0.
\end{equation}

In the counting argument, the Ky Fan sum $\Phi_m(L_{r,k})$ bounds the motion of the selected eigenphases, while $\tau_m(L_{r,k})$ remains strictly positive. The resulting quantitative estimate is proved in Proposition~\ref{d:prop:remainder}.

It remains to establish two links: an actual Dirichlet resonance must produce unitary phases approaching the reference point $1$ with the correct multiplicity, and the positivity of $L_{r,k}$ must then be converted into a quantitative lower bound for their total phase motion. Sections~\ref{d:sec:dtn-pole} and~\ref{d:sec:dirichlet-branch} establish the first link, while Section~\ref{d:sec:crossing-count} establishes the second.

\section{Dirichlet resonances and the Dirichlet-to-Neumann map}
\label{d:sec:dtn-pole}

\subsection{Negative divergence at Dirichlet resonances}
 Recall from Section~\ref{d:sec:transport} that the Hilbert-space transport is realized on Herglotz waves, whereas the actual Dirichlet eigenspace enters through boundary spectral data. We first show that, as a resonance is approached from below, the Dirichlet-to-Neumann quadratic form diverges to $-\infty$ on a space of boundary data whose dimension equals the resonance multiplicity. Section~\ref{d:sec:dirichlet-branch} then relates this divergence to eigenphases of the unitary family $U_r$ using Herglotz approximation and the Cayley transform.

We continue with $
\Omega_r=F(r\D),
\gamma_r(\theta)=F(re^{i\theta})$,
and keep the energy $k^2>0$ fixed.  We call $r$ \emph{nonresonant} if
\[
k^2\notin\sigma(-\Delta_{\Omega_r}^D).
\]
For a nonresonant radius $r$, the Dirichlet problem
\[
(\Delta+k^2)u=0\quad\text{in }\Omega_r,
\qquad
u\circ\gamma_r=h,
\]
 is uniquely solvable for
every $h\in H^{1/2}(\T)$, and the corresponding Dirichlet-to-Neumann
map  $
\DD_r:H^{1/2}(\T)\longrightarrow H^{-1/2}(\T)$ defined by
\[
\DD_rh
=
|\gamma_r'|\,
\partial_\nu u\circ\gamma_r
\]
is bounded. The calligraphic symbol $\DD_r$ denotes the Dirichlet-to-Neumann map and should not be confused with the radial transport coefficient $D_r(\theta)$ used in Section~\ref{d:sec:geometry}. Here the factor $|\gamma_r'|$ expresses the normal derivative as a density
with respect to $d\theta$.
Moreover, Green's identity gives
\[
\langle h_1,\DD_r h_2\rangle
=
\overline{\langle h_2,\DD_r h_1\rangle},
\qquad
h_1,h_2\in H^{1/2}(\T),
\]
where the brackets denote the natural
$H^{1/2}$--$H^{-1/2}$ boundary duality.
See \cite{LM,McLean}.

Under the conformal change of variables $\widetilde u(z)=u(F(rz))$,
the Helmholtz equation on $\Omega_r$ becomes
\[
\bigl(-\Delta-k^2w_r\bigr)\widetilde u=0
\qquad\text{in }\D,
\]
where
\[
w_r(z)=r^2|F'(rz)|^2.
\]
Thus the pulled-back differential operator on the fixed disk is the Schr\"odinger operator 
\begin{equation*}
    \mathscr A_r=-\Delta-k^2w_r.
\end{equation*}
Let $A_r$ denote the Dirichlet realization of $\mathscr A_r$ on
$L^2(\D)$, with domain $\Dom A_r=H^2(\D)\cap H_0^1(\D)$. 
Since $F(re^{i\theta})=\gamma_r(\theta)$,
we have $
\widetilde u(e^{i\theta})
=
u(\gamma_r(\theta))=h(\theta) $ and thus
\[
\partial_{\nu_{\D}}\widetilde u(e^{i\theta})
=r|F'(re^{i\theta})| \partial_{\nu_{\Omega_r}}u(\gamma_r(\theta))
=
|\gamma_r'(\theta)|
\partial_{\nu_{\Omega_r}}u(\gamma_r(\theta)).
\]
Therefore our normalization of $\DD_r$ is chosen precisely so that,
after pullback to $\D$,
\[
\DD_r h
=
\partial_{\nu_{\D}}\widetilde u|_{\T}.
\]
Thus $\DD_r$ is exactly the Dirichlet-to-Neumann map on $\T$
associated with the pulled-back operator
$\mathscr A_r=-\Delta-k^2w_r$.

The point of passing to the fixed disk is that, for a resonant radius $r_*$,  a Dirichlet eigenvalue
of $\Omega_{r_*}$ at energy $k^2$ becomes a zero eigenvalue of
$A_{r_*}$.  As $r \to r_*-$ from below, the corresponding
eigenvalues of $A_r$ approach zero from the positive side.  The inverse
$A_r^{-1}$ therefore becomes large on precisely this finite-dimensional
spectral subspace.  The next lemma shows that this singularity is
visible on the boundary, with the correct sign and multiplicity.

\begin{lemma}
    \label{d:lem:pole}
    Suppose that $k^2$ is a Dirichlet eigenvalue of $\Omega_{r_*}$ of
    multiplicity $m$, where $r_*>0$.  Then there exist linearly independent
    functions $
    h_1,\ldots,h_m\in C^\infty(\T)$
    and numbers $M_r\to+\infty$ as $r \to r_*-$ such that, for every
    sufficiently close nonresonant $r<r_*$, and for every $c\in\mathbb C^m$, the function $h_c=\sum_{j=1}^m c_jh_j$ satisfies
    \begin{equation}\label{d:eq:pole}
        \langle h_c,\DD_rh_c\rangle
        \le -M_r|c|^2.
    \end{equation}
\end{lemma}

\begin{proof}
    We separate the argument into four steps.

    \medskip
    \noindent
    \emph{Step 1: locate the spectral cluster approaching zero.}
    The Dirichlet eigenvalues
    \[
    r\longmapsto\lambda_j(\Omega_r)
    \]
    are continuous and strictly decreasing for $r>0$ by the standard Dirichlet domain-monotonicity and continuity theory; see, for example, \cite{Davies}.  We use these classical facts here and again in the final conformal-exhaustion argument without reproving them.

    By conformal invariance of the Dirichlet energy, the Dirichlet eigenvalues of $\Omega_r$ admit the weighted min--max
    characterization
    \[
    \lambda_j(\Omega_r)
    =
    \min_{\substack{E\subset H_0^1(\D)\\ \dim E=j}}
    \max_{0\ne v\in E}
    \frac{\int_\D |\nabla v|^2}
    {\int_\D w_r|v|^2}.
    \]
    On the other hand, the quadratic form of $
    A_r=-\Delta-k^2w_r$
    is
    \[
    q_r[v]
    =
    \int_\D |\nabla v|^2
    -
    k^2\int_\D w_r|v|^2.
    \]
    Hence $q_r[v]<0$ precisely when the corresponding weighted Rayleigh
    quotient is smaller than $k^2$.  The min--max principle therefore gives
    \[
    \#\{\text{negative eigenvalues of }A_r\}
    =
    \#\{j:\lambda_j(\Omega_r)<k^2\},
    \]
    and similarly
    \[
    \dim\ker A_r
    =
    \dim\ker(-\Delta_{\Omega_r}^D-k^2).
    \]
    Thus $\dim\ker A_{r_*}=m$.

    The potentials $w_r$ depend continuously on $r$ in
    $L^\infty(\D)$.  Since
    \[
    A_r-A_{r_*}=-k^2(w_r-w_{r_*})
    \]
    is therefore a small bounded perturbation of $A_{r_*}$, the resolvent
    identity implies norm-resolvent continuity of the family $A_r$.
    At $r=r_*$, the point $0$ is an isolated eigenvalue of $A_{r_*}$ of
    multiplicity $m$.  Choose a small positively oriented circle
    $\Gamma$ around $0$ which contains no other point of
    $\sigma(A_{r_*})$.  For $r$ sufficiently close to $r_*$, the same
    circle lies in the resolvent set of $A_r$, and we define the Riesz
    projection
    \[
    P_r
    =
    \frac{1}{2\pi i}
    \int_\Gamma (z-A_r)^{-1}\,dz .
    \]
    At $r=r_*$ this projection is the orthogonal projection $
    P_*:\ L^2(\D)\longrightarrow\ker A_{r_*}$.
    Norm-resolvent continuity gives
    \[
    \|P_r-P_*\|\longrightarrow0
    \qquad(r\to r_*).
    \]
    In particular, for $r$ sufficiently close to $r_*$,
    \[
    \rank P_r=\rank P_*=m.
    \]
    Thus $\ran P_r$ is precisely the $m$-dimensional spectral subspace of
    $A_r$ corresponding to the eigenvalues which remain close to zero.

   Strict monotonicity determines the sign of the small cluster through the Morse index, without identifying individual eigenvalues of $A_r$ with weighted Dirichlet eigenvalue branches. For a self-adjoint operator with discrete spectrum
bounded below, let $n_-(A)$ denote the number of its
negative eigenvalues, counted with multiplicity. Let
     \[
     q:=\#\{j:\lambda_j(\Omega_{r_*})<k^2\}=n_-(A_{r_*}).
     \]
    For all $r<r_*$ sufficiently close to $r_*$, strict domain monotonicity gives exactly the same number $q$ of Dirichlet eigenvalues below $k^2$; hence $n_-(A_r)=q$. By the choice of the Riesz circle $\Gamma$ and norm-resolvent continuity, the spectrum of $A_r$ outside $\ran P_r$ stays away from zero and carries these same $q$ negative eigenvalues. Therefore the $m$ eigenvalues in $\ran P_r$ cannot be negative. Since the radii under consideration are nonresonant, they cannot vanish either. Thus the whole Riesz cluster is positive. Since it converges to the zero eigenspace of $A_{r_*}$, there exists $\varepsilon_r>0$ with $\varepsilon_r\to0$ such that
     \begin{equation}\label{d:eq:small-positive-cluster}
          0<A_r|_{\ran P_r}\le \varepsilon_r I
     \end{equation}
      for all sufficiently close nonresonant $r<r_*$. Equivalently, every eigenvalue of $A_r$ on $\ran P_r$ lies in $(0,\varepsilon_r]$.

     By construction, $P_r$ contains the entire spectral cluster near
     zero.  The remaining spectrum stays a fixed positive distance away
     from zero when $r$ is close to $r_*$.  Hence there exists
     $\delta>0$, independent of such $r$, for which $\sigma\left(A_r|_{\ran(I-P_r)}\right)\cap(-\delta,\delta)=\varnothing$. Consequently,
     \begin{equation}\label{d:eq:complement-inverse}
         \bigl\|A_r^{-1}(I-P_r)\bigr\|
         \le \frac1\delta .
     \end{equation}
     Thus the only part of $A_r^{-1}$ that can diverge as
     $r \to r_*-$ is its restriction to the $m$-dimensional space
     $\ran P_r$.

    \medskip
    \noindent
    \emph{Step 2: choose $m$ boundary data which detect the small
        spectral cluster.}
    Choose an orthonormal basis $
    \phi_1,\ldots,\phi_m$
    of $\ker A_{r_*}$.  Since each $\phi_j$ satisfies the Dirichlet
    condition on $\T$, its nontrivial boundary information is carried by
    its normal derivative.  Set
    \[
    g_j:=\partial_\nu\phi_j|_{\T},
    \qquad j=1,\ldots,m.
    \]
    Elliptic regularity gives $g_j\in C^\infty(\T)$.

    The functions $g_1,\ldots,g_m$ are linearly independent.  Indeed,
    suppose that $
    \sum_{j=1}^m c_j g_j=0$
    and set $\phi=\sum_{j=1}^m c_j\phi_j$.
    Then $\phi$ satisfies
    \[
    A_{r_*}\phi=0,
    \qquad
    \phi|_{\T}=0,
    \qquad
    \partial_\nu\phi|_{\T}=0.
    \]
    After transferring $\phi$ back to $\Omega_{r_*}$, we obtain a
    Helmholtz solution with vanishing Dirichlet and Neumann data on
    $\partial\Omega_{r_*}$.  Uniqueness for the Cauchy problem (equivalently, boundary unique continuation; see, e.g., \cite{LM,McLean}) therefore
    implies $\phi=0$, and hence $c_1=\cdots=c_m=0$.

    We now use these normal derivatives as boundary test data for the
    Dirichlet-to-Neumann map.  Accordingly, set
    \[
    h_j:=g_j,
    \qquad
    h_c:=\sum_{j=1}^m c_jh_j.
    \]
    The point of this choice is that the boundary data $h_j$ couple
    nontrivially, and with full multiplicity, to the eigenspace
    $\ker A_{r_*}$.  This coupling will be quantified in Step~4 through
    the Gram matrix of the functions $g_j$.

    \medskip
    \noindent
    \emph{Step 3: express the Dirichlet-to-Neumann form through
        $A_r^{-1}$.}
    Fix a sufficiently close nonresonant radius $r<r_*$, so that
    $A_r$ is invertible.  Choose smooth extensions
    \[
    \mathcal E h_j\in C^\infty(\overline\D),
    \qquad
    (\mathcal E h_j)|_{\T}=h_j,
    \]
    and extend $\mathcal E$ linearly to $
    \mathcal H_*:=\operatorname{span}\{h_1,\ldots,h_m\}$.
    For $h\in\mathcal H_*$, set
    \[
    f_{r,h}:=\mathscr A_r\mathcal Eh,
    \qquad
    v_{r,h}:=A_r^{-1}f_{r,h},
    \]
    and define
    \[
    u_{r,h}:=\mathcal Eh-v_{r,h}.
    \]
    Since $v_{r,h}\in\Dom A_r\subset H_0^1(\D)$, it has zero boundary
    trace.  Moreover,
    \[
    \mathscr A_rv_{r,h}=f_{r,h}
    =\mathscr A_r\mathcal Eh.
    \]
    Consequently
    \[
    \mathscr A_ru_{r,h}=0
    \quad\text{in }\D,
    \qquad
    u_{r,h}|_{\T}=h.
    \]
    Thus $u_{r,h}$ is precisely the solution of the pulled-back Dirichlet
    problem with boundary value $h$.

   With the convention fixed above, let
    \[
    q_r[p,q]
    :=
    \int_\D
    \left(
    \nabla\overline p\cdot\nabla q
    -
    k^2w_r\,\overline p\,q
    \right)\,dA
    \]
    be the Hermitian sesquilinear form on $H^1(\D)$
extending the form associated with $A_r$ on $H_0^1(\D)$.
By Green's identity
    and the normalization of the Dirichlet-to-Neumann map,
    \[
    \langle h,\DD_rh\rangle
    =
    q_r[u_{r,h},u_{r,h}].
    \]
    Since $v_{r,h}\in H_0^1(\D)$ and
    $\mathscr A_ru_{r,h}=0$, the weak equation for $u_{r,h}$ gives $
    q_r[v_{r,h},u_{r,h}]=0$.
    By the Hermitian symmetry of $q_r$, also $
    q_r[u_{r,h},v_{r,h}]=0$.
     Using $
    \mathcal Eh=u_{r,h}+v_{r,h}$,
    we therefore obtain
    \[
    q_r[\mathcal Eh,\mathcal Eh]
    =
    q_r[u_{r,h},u_{r,h}]
    +
    q_r[v_{r,h},v_{r,h}].
    \]
    Moreover, $
    q_r[v_{r,h},v_{r,h}]
    =
    \langle f_{r,h},A_r^{-1}f_{r,h}\rangle$.
    Hence
    \begin{equation}\label{d:eq:schur}
        \langle h,\DD_rh\rangle
        =
        q_r[\mathcal Eh,\mathcal Eh]
        -
        \langle f_{r,h},A_r^{-1}f_{r,h}\rangle.
    \end{equation}

    This identity separates a uniformly bounded term from the possible
    singularity.  Indeed, $\mathcal H_*$ is finite dimensional, the
    extensions $\mathcal Eh$ are fixed, and $w_r$ remains uniformly
    bounded for $r$ near $r_*$.  Hence
    \[
    |q_r[\mathcal Eh,\mathcal Eh]|
    \le C\|h\|^2
    \]
    uniformly in such $r$.  The only term that can diverge as
    $r\to r_*-$ is therefore
    \[
    \langle f_{r,h},A_r^{-1}f_{r,h}\rangle,
    \]
    through the small positive spectral cluster of $A_r$ identified in    Step~1.

    \medskip
    \noindent
    \emph{Step 4: the small positive cluster yields negative divergence of the Dirichlet-to-Neumann quadratic form.}
    We first show that the boundary data chosen in Step~2 couple
    uniformly to all $m$ directions of the small spectral cluster.
    At $r=r_*$, Green's identity and
    $A_{r_*}\phi_i=0$ give
    \begin{equation}\label{d:eq:cluster-coupling}
        \langle\phi_i,f_{r_*,h}\rangle_{L^2(\D)}
        =
        \int_\T \overline{g_i}\,h\,d\theta.
    \end{equation}
    Indeed, $f_{r_*,h}=\mathscr A_{r_*}\mathcal Eh$ and
    $\phi_i|_{\T}=0$, so the interior terms cancel and only the normal
    derivative of $\phi_i$ remains on the boundary.

    For $h=h_c=\sum_{j=1}^m c_jg_j$, equation \eqref{d:eq:cluster-coupling} becomes
    \[
    \langle\phi_i,f_{r_*,h_c}\rangle
    =
    \sum_{j=1}^m G_{ij}c_j,
    \qquad
    G_{ij}:=\int_\T\overline{g_i}g_j\,d\theta.
    \]
    The matrix $G$ is the Gram matrix of the linearly independent
    functions $g_1,\ldots,g_m$, and is therefore positive definite.
    Since $\phi_1,\ldots,\phi_m$ form an orthonormal basis of
    $\ran P_*=\ker A_{r_*}$,
    \[
    \|P_*f_{r_*,h_c}\|_{L^2(\D)}^2
    =
    \sum_{i=1}^m
    \left|
    \langle\phi_i,f_{r_*,h_c}\rangle
    \right|^2
    =
    |Gc|^2.
    \]
    Hence there is a constant $c_*>0$ such that for any $c\in\mathbb C^m$,
    \[
    \|P_*f_{r_*,h_c}\|_{L^2(\D)}
    \ge c_*|c|.
    \]
    Since $P_r\to P_*$ in operator norm and
    $f_{r,h_c}\to f_{r_*,h_c}$ uniformly for $|c|=1$, after shrinking
    the neighborhood of $r_*$ we obtain a uniform constant $c_0>0$ such
    that
    \begin{equation}\label{d:eq:cluster-lower-bound}
        \|P_rf_{r,h_c}\|_{L^2(\D)}
        \ge c_0|c|
    \end{equation}
    for every $c\in\mathbb C^m$ and every sufficiently close
    $r<r_*$.

    We now estimate the singular term in \eqref{d:eq:schur}.
    Since $P_r$ is a spectral projection of $A_r$, it commutes with
    $A_r^{-1}$.  Therefore
    \begin{align*}
        \langle f_{r,h_c},A_r^{-1}f_{r,h_c}\rangle
         =
        \langle P_rf_{r,h_c},
        A_r^{-1}P_rf_{r,h_c}\rangle +
        \langle (I-P_r)f_{r,h_c},
        A_r^{-1}(I-P_r)f_{r,h_c}\rangle.
    \end{align*}
    On $\ran P_r$, \eqref{d:eq:small-positive-cluster} implies
    \[
    A_r^{-1}|_{\ran P_r}
    \ge \varepsilon_r^{-1}I.
    \]
    Consequently, by \eqref{d:eq:cluster-lower-bound},
    \begin{align*}
        -\langle P_rf_{r,h_c},
        A_r^{-1}P_rf_{r,h_c}\rangle
         \le
        -\frac{1}{\varepsilon_r}
        \|P_rf_{r,h_c}\|^2  \le
        -\frac{c_0^2}{\varepsilon_r}|c|^2.
    \end{align*}

    On the complementary spectral subspace, the contribution is uniformly bounded.
    Indeed, \eqref{d:eq:complement-inverse} and the finite-dimensional
    dependence of $f_{r,h_c}$ on $c$ give
    \[
    \left|
    \langle (I-P_r)f_{r,h_c},
    A_r^{-1}(I-P_r)f_{r,h_c}\rangle
    \right|
    \le C_1|c|^2.
    \]
    Likewise, since the extensions $\mathcal Eh_c$ lie in a fixed
    finite-dimensional space and $w_r$ remains uniformly bounded near
    $r_*$,
    \[
    |q_r[\mathcal Eh_c,\mathcal Eh_c]|
    \le C_2|c|^2.
    \]

    Substituting these estimates into \eqref{d:eq:schur}, we obtain
    \[
    \langle h_c,\DD_rh_c\rangle
    \le
    -\frac{c_0^2}{\varepsilon_r}|c|^2
    +(C_1+C_2)|c|^2.
    \]
    Thus, with $
    M_r
    :=
    \frac{c_0^2}{\varepsilon_r}-(C_1+C_2)$,
    we have
    \[
    M_r\longrightarrow+\infty
    \qquad(r \to r_*-),
    \]
    and therefore
    \[
    \langle h_c,\DD_rh_c\rangle
    \le -M_r|c|^2
    \qquad\text{for every }c\in\mathbb C^m.
    \]
    This proves \eqref{d:eq:pole}.
    \end{proof}

    Thus an $m$-dimensional Dirichlet eigenspace at the resonant radius
    $r_*$ produces an $m$-dimensional space of boundary data on which
    the Dirichlet-to-Neumann quadratic form tends uniformly to
    $-\infty$ as $r \to r_*-$.  Indeed, the corresponding $m$ eigenvalues of $A_r$ approach $0$ from the positive side, so $A_r^{-1}$ becomes unbounded on their spectral subspace. The negative sign in \eqref{d:eq:schur} then gives divergence to $-\infty$ of the Dirichlet-to-Neumann quadratic form on the chosen boundary subspace. The next subsection relates this divergence to the unitary monodromy.

\subsection{Herglotz trace density and exclusion of fixed vectors}

The negative divergence of the Dirichlet-to-Neumann quadratic form was established using boundary data $h_1,\ldots,h_m$ constructed from the resonant eigenspace. By contrast, the Hilbert-space transport developed in
Section~\ref{d:sec:transport} is available only for Herglotz waves:
for these waves the normalized derivative sequence is known to belong to $\HH$.  To apply the transport identity to these boundary data, we use approximation by Herglotz traces. 

The first ingredient is that, away from the Dirichlet spectrum,
Herglotz traces are dense in the boundary space in which the
Dirichlet-to-Neumann map acts.  This will later allow us, at each
fixed nonresonant radius, to approximate the finitely many boundary
directions obtained in Lemma~\ref{d:lem:pole}. The following statement is a standard density result for Herglotz wave functions; see, for example, \cite{CK,ColtonSleeman2001,Weck2004}.

\begin{lemma}
    \label{d:lem:density}
    Let $G$ be a bounded smooth Jordan domain and suppose that
    $k^2\notin\sigma(-\Delta_G^D)$.  Then the boundary traces of the Herglotz waves
    \eqref{d:eq:herglotz} are dense in $H^{1/2}(\partial G)$.
\end{lemma}

\begin{proof}
Given $h\in H^{1/2}(\partial G)$, nonresonance gives a unique solution $u\in H^1(G)$ of $(\Delta+k^2)u=0$ with trace $h$.  The Herglotz--Runge approximation theorem gives Herglotz waves $u_n$ with $u_n\to u$ in $H^1(G)$; see \cite{CK,ColtonSleeman2001,Weck2004}.   Thus we use the standard $H^1$-density in the Helmholtz solution space, and the boundary statement follows from the trace theorem: continuity of the trace map $H^1(G)\to H^{1/2}(\partial G)$ then yields $u_n|_{\partial G}\to h$ in $H^{1/2}(\partial G)$.
\end{proof}

A smooth parametrization of $\partial G$ induces an isomorphism on
$H^{1/2}$, so the same density statement holds after pulling the
boundary traces back to $\T$.  In particular, it applies to
$\Omega_r$ whenever
\[
k^2\notin\sigma(-\Delta_{\Omega_r}^D).
\]
As above, such a radius $r$ will be called \emph{nonresonant}.

  The second ingredient is
that the same density statement rules out the spectral point $1$ as
an eigenvalue of the monodromy $V_r$.
As a result, $\ran(I-V_r)$ is dense, and the inverse
of $I-V_r$ is defined on that range, which is precisely
the natural domain of the Cayley operator introduced below.

\begin{lemma}
    \label{d:lem:injective}
    At every nonresonant radius $r$,
    \[
    \ker(I-V_r)=\{0\}.
    \]
\end{lemma}

\begin{proof}
    Fix a nonresonant radius and suppress the subscript $r$ for simplicity.
    Suppose that $Vv=v$ for some $v\in\HH$, and set
    \[
    y(\theta)=W(\theta)v.
    \]
    Then $y(2\pi)=y(0)=v$.

    We use the standard Fourier realization of
    $\HH=\ell^2(\mathbb N_0)$.  For each $\theta$, put
    \[
    Y(\theta,\phi)
    :=
    \sum_{n\ge0}y_n(\theta)e^{in\phi}
    \in L^2\!\left(\T,\frac{d\phi}{2\pi}\right).
    \]
    Under this identification,
    \[
    Sy\longleftrightarrow e^{i\phi}Y,
    \qquad
    S^*y\longleftrightarrow
    e^{-i\phi}(Y-y_0).
    \]
    Hence the homogeneous transport equation
    \[
    y'
    =
    -\frac{ik}{2}
    \bigl(\gamma'S^*+\overline{\gamma'}S\bigr)y
    \]
    becomes the following $L^2(\T)$-valued
    first-order differential equation in $\theta$:
    \begin{equation*}
        \partial_\theta Y
        =
        -ik\,\operatorname{Re}
        \bigl(\gamma'(\theta)e^{-i\phi}\bigr)Y
        +
        \frac{ik}{2}\gamma'(\theta)e^{-i\phi}y_0(\theta).
    \end{equation*}

After introducing the integrating factor $
    E(\theta,\phi)
    :=
    \exp\left[
    ik\,\operatorname{Re}
    \bigl((\gamma(\theta)-\gamma(0))e^{-i\phi}\bigr)
    \right]$, we get

    \[
    \partial_\theta(EY)
    =
    \frac{ik}{2}
    e^{-i\phi}\gamma'(\theta)y_0(\theta)E(\theta,\phi).
    \]
    Integrating from $0$ to $2\pi$, and using   $\gamma(2\pi)=\gamma(0)$ and $Y(2\pi,\phi)=Y(0,\phi)$,
    we obtain
    \[
    \int_0^{2\pi}
    \gamma'(\theta)y_0(\theta)
    \exp\left[
    ik\,\operatorname{Re}
    \bigl((\gamma(\theta)-\gamma(0))e^{-i\phi}\bigr)
    \right]\,d\theta
    =0.
    \]
    Multiplying by the  nonzero phase $\exp\left[
    ik\,\operatorname{Re} \bigl(\gamma(0)e^{-i\phi}\bigr)
    \right]$
    yields
    \begin{equation}\label{d:eq:fixed-plane-wave}
        \int_0^{2\pi}
        \gamma'(\theta)y_0(\theta)
        e^{ik\operatorname{Re}(\gamma(\theta)e^{-i\phi})}\,d\theta
        =0.
    \end{equation}
    The identity initially holds for almost every $\phi$, and hence for
    every $\phi$ by continuity.

    Replacing $\phi$ by $\phi+\pi$ in
    \eqref{d:eq:fixed-plane-wave} gives
    \[
    \int_0^{2\pi}
    \gamma'(\theta)y_0(\theta)
    e^{-ik\operatorname{Re}(\gamma(\theta)e^{-i\phi})}\,d\theta
    =0
    \qquad(\phi\in\T).
    \]
    These exponentials are precisely the plane waves occurring in the
    Herglotz representation \eqref{d:eq:herglotz}.  Therefore the
    continuous functional
    \[
    \Lambda(h)
    :=
    \int_0^{2\pi}
    \gamma'(\theta)y_0(\theta)h(\theta)\,d\theta
    \]
    annihilates every Herglotz trace.  By
    Lemma~\ref{d:lem:density}, such traces are dense in
    $H^{1/2}(\T)$ at a nonresonant radius.  Hence $\Lambda=0$,
    and consequently
    \[
    \gamma'(\theta)y_0(\theta)=0.
    \]
    Since $\gamma$ is a regular parametrization,
    $\gamma'(\theta)\ne0$, so $
    y_0(\theta)=0$ for all $\theta$.
    Lemma~\ref{d:lem:zeroth-observation} now gives
    $y\equiv0$, and hence $v=y(0)=0$.
\end{proof}

Since $V_r$ is unitary, Lemma~\ref{d:lem:injective} also gives
\[
\overline{\ran(I-V_r)}
=
\ker(I-V_r^*)^\perp
=
\HH.
\]
Thus $(I-V_r)^{-1}$ is well defined on the dense subspace
$\ran(I-V_r)$.  This is precisely the domain needed for the Cayley transform
below.

\section{The Cayley transform and resonance phases}
\label{d:sec:dirichlet-branch}

\subsection{The Cayley transform}
The Dirichlet-to-Neumann map is self-adjoint, whereas the monodromy
$V_r$ is unitary.  In order to compare their spectra, it is convenient
to replace the unit circle by a real spectral coordinate.  The Cayley transform is the standard passage from a unitary
operator to a self-adjoint operator.
We record the precise statement. For background, see \cite{Kato} and \cite[Section~13.1]{Schmudgen}.

 \begin{lemma}
     \label{d:lem:cayley}
     Let $V$ be unitary on a Hilbert space $\HH$.  Assume $V-I$ is compact and
     $\ker(I-V)=\{0\}$, and set
     \[
     \mathcal K(V):=i(I+V)(I-V)^{-1},
     \qquad
     \Dom\mathcal K(V):=\ran(I-V).
     \]
     Then $\Dom\mathcal K(V)$ is dense in $\HH$,
     $\mathcal K(V)$ is self-adjoint with compact resolvent.

     If $U=V^*$ and $e^{i\theta}$, $0<\theta<2\pi$, is an eigenvalue of $U$,
     then $\cot(\theta/2)$ is an eigenvalue of $\mathcal K(V)$ with the
     same multiplicity.
 \end{lemma}

\begin{proof}
Density of $\ran(I-V)$ follows from $\ker(I-V^*)=\ker(I-V)=\{0\}$.  This is the standard inverse Cayley transform of $V$; in particular it is self-adjoint.  One may also see the resolvent directly: on $\HH$,
\[
(\mathcal K(V)+i)^{-1}=\frac{I-V}{2i}.
\]
Hence the resolvent is compact because $V-I$ is compact.  Finally, if $Vv=e^{-i\theta}v$, then
\[
\mathcal K(V)v=i\frac{1+e^{-i\theta}}{1-e^{-i\theta}}v=\cot(\theta/2)v,
\]
and the eigenspace is unchanged.
\end{proof}

 At a nonresonant radius we now apply the lemma to $V_r$ and write
 \begin{equation}\label{d:eq:cayley}
     \KK_r
     :=
     i(I+V_r)(I-V_r)^{-1},
     \qquad
     \Dom\KK_r=\ran(I-V_r).
 \end{equation}
Lemma~\ref{d:lem:injective} gives
$\ker(I-V_r)=\{0\}$,
while Proposition~\ref{d:prop:radial} gives
$V_r-I\in\mathcal S_1$.
In particular, $V_r-I$ is compact.
Thus the Cayley transform \eqref{d:eq:cayley}
satisfies all the hypotheses of Lemma~\ref{d:lem:cayley}: $\KK_r$ is a densely defined self-adjoint operator with compact
 resolvent.

 Recall from \eqref{d:eq:Ur-definition} that
$U_r=R_r^*V_r^*R_r$.
 Hence $U_r$ and $V_r^*$ are unitarily equivalent and have the same
 eigenvalues with the same multiplicities.  Consequently, for
 $0<\theta<2\pi$,
 \begin{equation}\label{d:eq:cot}
     e^{i\theta}\in\sigma(U_r)
     \quad\Longleftrightarrow\quad
     \cot(\theta/2)\in\sigma(\KK_r),
 \end{equation}
 with equal multiplicities.  Since $U_r-I$ is compact, every spectral
 point $e^{i\theta}\ne1$ appearing here is in fact an isolated
 eigenvalue of finite multiplicity.

 In particular,
 \[
 \theta \to2\pi-
 \quad\Longleftrightarrow\quad
 \cot(\theta/2)\longrightarrow-\infty.
 \]
 Thus, in order to prove that a Dirichlet resonance produces $m$
 phases of $U_r$ approaching $1$ from the $2\pi$ side, it is enough
 to prove that $\KK_r$ develops at least $m$ eigenvalues tending to
 $-\infty$.  The next step derives the boundary identity that transfers
 the negative divergence of the Dirichlet-to-Neumann quadratic form to the Cayley operator. 

 \subsection{The Cayley-DtN identity}
 \label{d:sec:boundary}

 We now connect the Hilbert-space transport to the
 Dirichlet-to-Neumann map.  The key point is that the normal derivative
 of a Helmholtz solution can be recovered from its boundary value and
 the zeroth component of the normalized derivative sequence.

 Fix a radius $r>0$, write $\gamma=\gamma_r$, and let $h=u\circ\gamma$ as before.
 For a positively oriented boundary curve, the outward unit normal is $
 \nu=-i\,\frac{\gamma'}{|\gamma'|}$. Hence the tangential and normal derivatives satisfy
 \begin{align}
     h'
     &=
     \gamma'\partial u+\overline{\gamma'}\,\bar\partial u,
     \label{d:eq:tangential-derivative}\\
     |\gamma'|\partial_\nu u
     &=
     -i\gamma'\partial u
     +i\overline{\gamma'}\,\bar\partial u.
     \label{d:eq:normal-derivative}
 \end{align}
 For a Herglotz wave $u=u_a$, let
 \[
 y_n
 =
 \left(\frac{2i}{k}\right)^{n+1}\partial^{n+1}u,
 \qquad n\ge0,
 \]
 denote its square-summable normalized derivative sequence.  Since $
 y_0=\frac{2i}{k}\partial u$,
 subtracting the appropriate multiple of
 \eqref{d:eq:tangential-derivative} from
 \eqref{d:eq:normal-derivative} gives the elementary identity
 \begin{equation}\label{d:eq:normal-from-y0}
     |\gamma'|\partial_\nu u
     =
     ih'-k\gamma' y_0.
 \end{equation}
 Thus, in order to recover the Dirichlet-to-Neumann map from the
 transport equation, we only need to understand the quantity
 $\gamma' y_0$.

 The variation-of-constants formula  \eqref{d:eq:variation-first} for the driven transport gives
 \[
 y(\theta)
 =
 W_r(\theta)v
 -
 \frac{ik}{2}W_r(\theta)
 \int_0^\theta
 W_r(s)^*e_0\,
 \overline{\gamma_r'(s)}h(s)\,ds,
 \]
 where $v=y(0)\in\HH$.
 Taking the zeroth component and multiplying by $\gamma_r'(\theta)$
 shows that $\gamma_r'y_0$ naturally splits into two pieces: one
 coming from the transported initial vector $v$, and one coming from
 the accumulated boundary forcing $h$.  This motivates the following
 two operators.

 Define the operator $\OO_r:\HH\longrightarrow L^2(\T,d\theta)$
 by
 \begin{equation}\label{d:eq:observation}
     (\OO_rv)(\theta)
     :=
     \gamma_r'(\theta)
     \langle e_0,W_r(\theta)v\rangle,
 \end{equation}
 and define the operator
 $\TD_r: L^2(\T,d\theta)\longrightarrow L^2(\T,d\theta)$ by
 \begin{equation}\label{d:eq:volterra-definition}
     (\TD_rh)(\theta)
     :=
     \int_0^\theta
     \gamma_r'(\theta)\overline{\gamma_r'(s)}
     \langle e_0,
     W_r(\theta)W_r(s)^*e_0\rangle
     h(s)\,ds.
 \end{equation}
Since $y_0(\theta)=\langle e_0,y(\theta)\rangle$,
the definitions \eqref{d:eq:observation} and
\eqref{d:eq:volterra-definition} give
 \begin{equation}\label{d:eq:y0-observation}
     \gamma_r' y_0
     =
     \OO_rv-\frac{ik}{2}\TD_rh.
 \end{equation}

 Using our convention that inner products are conjugate-linear
 in the first variable,
 \[
 \begin{aligned}
     \langle \OO_rv,h\rangle_{L^2(\T)}
      =
     \int_0^{2\pi}
     \overline{\gamma_r'(s)}
     \langle W_r(s)v,e_0\rangle h(s)\,ds =
     \left\langle
     v,
     \int_0^{2\pi}
     W_r(s)^*e_0\,\overline{\gamma_r'(s)}h(s)\,ds
     \right\rangle_{\HH}.
 \end{aligned}
 \]
  Hence the adjoint of $\OO_r$ is
 \begin{equation}\label{d:eq:Ostar}
     \OO_r^*h
     =
     \int_0^{2\pi}
     W_r(s)^*e_0\,
     \overline{\gamma_r'(s)}h(s)\,ds,
 \end{equation}
 and $\OO_r$ is Hilbert--Schmidt, with
 \[
 \|\OO_r\|_{\mathcal S_2}^2
 =
 \int_0^{2\pi}|\gamma_r'(\theta)|^2\,d\theta.
 \]
 Here $\mathcal S_2$ denotes the Hilbert--Schmidt class.
 Indeed, the last identity follows from the unitarity of $W_r$ and
 \[
 \sum_{n\ge0}
 \bigl|
 \langle e_0,W_r(\theta)e_n\rangle
 \bigr|^2
 =
 \|W_r(\theta)^*e_0\|^2
 =
 1.
 \]

 There is also a useful algebraic relation between the two operators.
 The integral kernel of $\OO_r\OO_r^*$ is
 \[
 K_r(\theta,s)
 =
 \gamma_r'(\theta)\overline{\gamma_r'(s)}
 \langle e_0,
 W_r(\theta)W_r(s)^*e_0\rangle.
 \]
 The operator $\TD_r$ has this kernel on the triangle $s<\theta$,
 whereas $\TD_r^*$ has the same kernel on the complementary triangle
 $\theta<s$.  Since the diagonal has measure zero,
 \begin{equation}\label{d:eq:Volterra}
     \TD_r+\TD_r^*=\OO_r\OO_r^*.
 \end{equation}
 For $r$ in a compact subinterval of $(0,1]$, the operators
 $\OO_r$ and $\TD_r$ are bounded with locally uniform norms.

 We are now ready to state the identity that connects the
 Dirichlet-to-Neumann map with the Cayley operator $\KK_r$.

 \begin{proposition} 
     \label{d:prop:identity}
     Let $r$ be nonresonant and let $
     h=u_a\circ\gamma_r$
     be the boundary trace of a Herglotz wave.  Then $
     \OO_r^*h\in\Dom\KK_r=\ran(I-V_r)$,  and     \begin{equation}\label{d:eq:identity}
         \DD_rh
         =
         i\partial_\theta h
         +
         \frac{ik^2}{4}(\TD_r-\TD_r^*)h
         +
         \frac{k^2}{4}\OO_r\KK_r\OO_r^*h.
     \end{equation}
     The identity \eqref{d:eq:identity} is understood in $H^{-1/2}(\T)$. For Herglotz traces the Volterra and Cayley terms are represented by $L^2$ functions, while $i\partial_\theta:H^{1/2}(\T)\to H^{-1/2}(\T)$.
     Moreover, if we define the quadratic form $\mathfrak r_r$ on $H^{1/2}(\T)$ by
     \begin{equation}\label{d:eq:remainder-form}
         \mathfrak r_r[h]
         :=
         \langle h,i\partial_\theta h\rangle
         +
         \frac{ik^2}{4}
         \langle h,(\TD_r-\TD_r^*)h\rangle,
         \qquad h\in H^{1/2}(\T).
     \end{equation}
    Here the first pairing in \eqref{d:eq:remainder-form} is the $H^{1/2}$--$H^{-1/2}$ duality pairing (and agrees with the $L^2$ pairing for smooth $h$). Then $\mathfrak r_r$ is real-valued, and for every compact interval
     $I\subset(0,1]$ there exists $C_I>0$ such that
     \begin{equation}\label{d:eq:remainder}
         |\mathfrak r_r[h]|
         \le
         C_I\|h\|_{H^{1/2}(\T)}^2,
         \qquad
         r\in I,\quad h\in H^{1/2}(\T).
     \end{equation}
 \end{proposition}

 \begin{proof}
     Let $u_a$ be a Herglotz wave with boundary trace $h$, and let
     $y=(y_n)_{n\ge0}$ be its normalized derivative sequence.  By
     \eqref{d:eq:jet-parseval}, $y(\theta)\in\HH$.  Put $v=y(0)$. Since $u_a$ is globally defined and
     $\gamma_r(0)=\gamma_r(2\pi)$, its derivative sequence is
     single-valued with $y(2\pi)=y(0)=v$.
     Evaluating the formula
     \eqref{d:eq:variation-first} at $\theta=2\pi$ and using
     \eqref{d:eq:Ostar} gives
     \begin{equation}\label{d:eq:endpoint}
         (I-V_r)v
         =
         -\frac{ik}{2}V_r\OO_r^*h.
     \end{equation}

     Hence
     \begin{equation*}
         \OO_r^*h
         =
         \frac{2i}{k}(V_r^*-I)v
         =
         \frac{2i}{k}(I-V_r)V_r^*v.
     \end{equation*}
     Therefore $\OO_r^*h\in\ran(I-V_r)=\Dom\KK_r$.

     We next compute the normal derivative.  Since $r$ is nonresonant,
     the restriction of $u_a$ to $\Omega_r$ is the unique solution of the
     Dirichlet problem with boundary value $h$, and therefore
     \[
     \DD_rh=|\gamma_r'|\partial_\nu u_a.
     \]
     Combining \eqref{d:eq:normal-from-y0} and
\eqref{d:eq:y0-observation} gives
     \begin{equation}\label{d:eq:DtN-before-Cayley}
         \DD_rh
         =
         i\partial_\theta h
         -
         k\OO_rv
         +
         \frac{ik^2}{2}\TD_rh.
     \end{equation}

     From \eqref{d:eq:endpoint}, $ v = -\frac{ik}{2} (I-V_r)^{-1}V_r\OO_r^*h$,
     where the inverse acts only on its natural range.  Substitution into
     \eqref{d:eq:DtN-before-Cayley} yields
     \[
     \DD_rh
     =
     i\partial_\theta h
     +
     \frac{ik^2}{2}
     \OO_r(I-V_r)^{-1}V_r\OO_r^*h
     +
     \frac{ik^2}{2}\TD_rh.
     \]
     On $\ran(I-V_r)$, we have $ i(I-V_r)^{-1}V_r = \frac12\KK_r-\frac{i}{2}I$. Consequently,
     \[
     \DD_rh
     =
     i\partial_\theta h
     +
     \frac{k^2}{4}\OO_r\KK_r\OO_r^*h
     -
     \frac{ik^2}{4}\OO_r\OO_r^*h
     +
     \frac{ik^2}{2}\TD_rh.
     \]
     Using \eqref{d:eq:Volterra},
     the last two terms combine as
     \[
     -\frac{ik^2}{4}(\TD_r+\TD_r^*)
     +
     \frac{ik^2}{2}\TD_r
     =
     \frac{ik^2}{4}(\TD_r-\TD_r^*).
     \]
     This proves \eqref{d:eq:identity}.

     It remains to estimate the first two terms in
     \eqref{d:eq:identity}.  They are precisely the terms grouped in
     $\mathfrak r_r$.  For $h(\theta)=\sum_{n\in\mathbb Z}h_ne^{in\theta}$, we have
     \[
     \bigl|
     \langle h,i\partial_\theta h\rangle
     \bigr|
     \le
     2\pi \sum_{n\in\mathbb Z}|n|\,|h_n|^2
     \le
     C\|h\|_{H^{1/2}(\T)}^2.
     \]
     Moreover, the operator $i(\TD_r-\TD_r^*)$ is bounded and self-adjoint on $L^2(\T)$, with operator norm locally
     uniform in $r>0$.  Hence
     \[
     \left|
     \frac{ik^2}{4}
     \langle h,(\TD_r-\TD_r^*)h\rangle
     \right|
     \le
     C_I\|h\|_{L^2(\T)}^2
     \le
     C_I\|h\|_{H^{1/2}(\T)}^2.
     \]
     This proves \eqref{d:eq:remainder}.
 \end{proof}

We remark that the estimate \eqref{d:eq:remainder} remains valid also at resonant radii, because $\mathfrak r_r$ depends only on the boundary
differentiation $i\partial_\theta$ and on the operator
$\TD_r$, which is constructed directly from the unitary transport
$W_r$.  Neither the Dirichlet resolvent $A_r^{-1}$ nor the Cayley
operator $\KK_r$ enters its definition.

 Formula \eqref{d:eq:identity} now separates the two parts that will
 play different roles near a Dirichlet resonance.  In quadratic-form
 language,
 \[
 \langle h,\DD_rh\rangle
 =
 \mathfrak r_r[h]
 +
 \frac{k^2}{4}
 \langle \OO_r^*h,\KK_r\OO_r^*h\rangle.
 \]
 The first term is uniformly bounded on bounded subsets of
 $H^{1/2}(\T)$, whereas the second contains the spectral information
 of the Cayley operator.  Thus a negative divergence of the Dirichlet-to-Neumann quadratic form must arise from the Cayley term.

 We emphasize that \eqref{d:eq:identity} is asserted only for Herglotz
 traces.  Although these traces are dense in $H^{1/2}(\T)$ at a
 nonresonant radius, density alone does not allow one to extend the
 Cayley term to arbitrary boundary data, because $\KK_r$ is generally
 unbounded.  In the next subsection we therefore approximate only the
 finitely many boundary directions from Lemma~\ref{d:lem:pole} by Herglotz traces and retain the
 domain condition $\OO_r^*h\in\Dom\KK_r$ throughout.

 \subsection{From Dirichlet resonances to unitary eigenphases} 
 \label{d:sec:poles}

 Lemma~\ref{d:lem:pole} gives an $m$-dimensional space of boundary
 data on which the Dirichlet-to-Neumann quadratic form tends uniformly
 to $-\infty$ as $r \to r_*-$.  Proposition~\ref{d:prop:identity}
 shows that, for Herglotz traces,
 \[
 \langle h,\DD_r h\rangle
 =
 \mathfrak r_r[h]
 +
 \frac{k^2}{4}
 \langle \OO_r^*h,\KK_r\OO_r^*h\rangle.
 \]
 The first term is uniformly bounded on bounded subsets of
 $H^{1/2}(\T)$. Consequently, the negative divergence must also appear in the Cayley quadratic form.

The only point requiring care is that the boundary data
$h_1,\ldots,h_m$ from Lemma~\ref{d:lem:pole} need not themselves be
Herglotz traces, whereas the
 identity above has been proved only for Herglotz traces.  We therefore
 approximate these finitely many directions at each nonresonant radius.

 \begin{proposition}
     \label{d:prop:near}
     Under the assumptions of Lemma~\ref{d:lem:pole}, let $\pi<\alpha<2\pi$. Then, for every sufficiently close nonresonant radius $r<r_*$,
     the unitary operator $U_r$ has at least $m$ eigenvalues, counted with
     multiplicity, whose arguments lie in $(\alpha,2\pi)$.
 \end{proposition}

 \begin{proof}
     Define
     \[
     J_0:\mathbb C^m\longrightarrow H^{1/2}(\T),
     \qquad
     J_0c= h_c:=\sum_{j=1}^m c_jh_j.
     \]
     By Lemma~\ref{d:lem:pole},
     \begin{equation}\label{d:eq:pole-J0}
         \langle J_0c,\DD_rJ_0c\rangle
         \le
         -M_r|c|^2,
         \qquad
         M_r\longrightarrow+\infty
         \quad(r \to r_*-).
     \end{equation}

     Fix for the moment a nonresonant $r<r_*$.  Since Herglotz traces
     are dense in $H^{1/2}(\T)$ by Lemma~\ref{d:lem:density}, we may
     choose Herglotz traces $\widetilde h_{r,1},\ldots,\widetilde h_{r,m}$ arbitrarily close to $h_1,\ldots,h_m$.  Set
     \[
     J_rc
     :=
     \sum_{j=1}^m c_j\widetilde h_{r,j}.
     \]
     Because the space $\mathbb C^m$ is finite dimensional, the
     approximations may be chosen so that
     \begin{equation}\label{d:eq:Jr-approx}
         \|J_r-J_0\|_{\mathbb C^m\to H^{1/2}(\T)}
         \le\delta_r,
     \end{equation}
     where $0<\delta_r\le1$ is chosen so that
     \begin{equation}\label{d:eq:delta-choice}
         \|\DD_r\|\,
         \delta_r\bigl(2\|J_0\|+1\bigr)
         \le1.
     \end{equation}

     In other words, although the norm $\|\DD_r\|$ may become large as $r \to r_*-$,  for each fixed nonresonant $r$, we may
     choose the Herglotz approximation such that $\delta_r$ is as small as we want.

We first show that the negative divergence in
\eqref{d:eq:pole-J0} persists under this approximation.
     Put
     \[
      e_c:=(J_r-J_0)c.
     \]
     Using the Hermitian symmetry of $\DD_r$,
     \begin{align*}
         \left|
         \langle J_rc,\DD_rJ_rc\rangle
         -
         \langle J_0c,\DD_rJ_0c\rangle
         \right|  \le
         \|\DD_r\|\,
         \|e_c\|_{H^{1/2}}
         \left(
         2\|J_0c\|_{H^{1/2}}
         +
         \|e_c\|_{H^{1/2}}
         \right).
     \end{align*}
     By \eqref{d:eq:Jr-approx}--\eqref{d:eq:delta-choice}, the right-hand
     side is at most $|c|^2$.  Hence
     \begin{equation}\label{d:eq:Herglotz-pole}
         \langle J_rc,\DD_rJ_rc\rangle
         \le
         -(M_r-1)|c|^2.
     \end{equation}
     Moreover,
     \begin{equation}\label{d:eq:Jr-bound}
         \|J_rc\|_{H^{1/2}(\T)}
         \le
         C_0|c|
     \end{equation}
     with $C_0$ independent of $r$ near $r_*$.

     Every $J_rc$ is again a Herglotz trace, being a finite linear
     combination of Herglotz traces.  We may therefore apply
     Proposition~\ref{d:prop:identity}.  Define
     \begin{equation*}
         x_{r,c}:=\OO_r^*J_rc.
     \end{equation*}
     By Proposition~\ref{d:prop:identity}, $x_{r,c}\in\Dom\KK_r$.
     Taking the quadratic form of \eqref{d:eq:identity}, we have
     \begin{equation}\label{d:eq:Cayley-quadratic}
         \langle J_rc,\DD_rJ_rc\rangle
         =
         \mathfrak r_r[J_rc]
         +
         \frac{k^2}{4}
         \langle x_{r,c},\KK_rx_{r,c}\rangle.
     \end{equation}

     Since $r_*>0$, we may restrict $r$ to a fixed compact interval
     contained in $(0,1]$ and containing all sufficiently close
     $r<r_*$.  By \eqref{d:eq:remainder} and
     \eqref{d:eq:Jr-bound},
     \[
     |\mathfrak r_r[J_rc]|
     \le C|c|^2
     \]
     with $C$ independent of such $r$.  Combining this with
     \eqref{d:eq:Herglotz-pole} and
     \eqref{d:eq:Cayley-quadratic} gives
     \begin{equation}\label{d:eq:Cayley-negative}
         \langle x_{r,c},\KK_rx_{r,c}\rangle
         \le
         -\frac{4}{k^2}(M_r-C_1)|c|^2
     \end{equation}
     for a constant $C_1$ independent of $r$.

     We also have a uniform upper bound for the norm of $x_{r,c}$.
     Indeed,
     \[
     \|x_{r,c}\|
     \le
     \|\OO_r^*\|\,\|J_rc\|_{L^2(\T)}
     \le
     C_2|c|,
     \]
     because $\OO_r$ is locally uniformly bounded in $r$ and
     $H^{1/2}(\T)\hookrightarrow L^2(\T)$.  Thus
     \begin{equation}\label{d:eq:x-bound}
         \|x_{r,c}\|\le C_2|c|.
     \end{equation}

     For $r$ sufficiently close to $r_*$ we have $M_r>C_1$.
     Then the map
     \[
     c\longmapsto x_{r,c}
     \]
     is injective.  Indeed, if $c\ne0$ but $x_{r,c}=0$, the left-hand side
     of \eqref{d:eq:Cayley-negative} would be zero whereas the right-hand
     side would be strictly negative, a contradiction.  Therefore
     \[
     E_r
     :=
     \{x_{r,c}:c\in\mathbb C^m\}
     \]
     is an $m$-dimensional subspace of $\Dom\KK_r$.

     Combining \eqref{d:eq:Cayley-negative} with
     \eqref{d:eq:x-bound}, every nonzero $x\in E_r$ satisfies
     \begin{equation}\label{d:eq:Cayley-Rayleigh}
         \frac{\langle x,\KK_rx\rangle}{\|x\|^2}
         \le
         -\frac{4(M_r-C_1)}{k^2C_2^2}.
     \end{equation}
     The right-hand side tends to $-\infty$ as $r \to r_*-$.

     We now convert this $m$-dimensional negative subspace into a spectral
     statement.  Fix $D>0$.  For $r$ sufficiently close to $r_*$,
     \eqref{d:eq:Cayley-Rayleigh} gives
     \[
     \langle x,\KK_rx\rangle<-D\|x\|^2
     \qquad
     (0\ne x\in E_r).
     \]
     Let $P_{r,D}:=\mathbf 1_{(-\infty,-D)}(\KK_r)$ be the spectral projection of $\KK_r$ onto $(-\infty,-D)$.
     We claim that
     \[
     \rank P_{r,D}\ge m.
     \]
     Otherwise $\rank P_{r,D}<m=\dim E_r$, so there would exist a nonzero vector $x\in E_r\cap\ker P_{r,D}$.
     The spectral support of such an $x$ lies in $[-D,\infty)$, and hence
     the spectral theorem gives
     \[
     \langle x,\KK_rx\rangle
     \ge
     -D\|x\|^2,
     \]
     contradicting the preceding strict inequality.  Thus $\KK_r$ has at
     least $m$ eigenvalues in $(-\infty,-D)$, counted with multiplicity.

     Finally choose $D>0$ so large that $
     -D<\cot(\alpha/2)$.
     By the Cayley--phase correspondence \eqref{d:eq:cot}, an eigenvalue
    $\lambda=\cot(\theta/2)<-D$ corresponds to an eigenphase satisfying
     $\theta>\alpha$,      because $\theta\mapsto\cot(\theta/2)$ is strictly decreasing on
     $(0,2\pi)$.  Hence $U_r$ has at least $m$ eigenvalues with arguments in $(\alpha,2\pi)$, counted with multiplicity.
 \end{proof}

The correspondence preserves multiplicity: an $m$-dimensional Dirichlet eigenspace gives $m$ independent boundary directions on which the Dirichlet-to-Neumann quadratic form diverges negatively; Herglotz approximation makes the Cayley identity applicable; and, for every $D>0$, the Cayley operator has at least $m$ eigenvalues below $-D$ whenever $r<r_*$ is sufficiently close to $r_*$, corresponding to eigenphases of $U_r$ approaching $2\pi$.

\section{Proof of the strict Dirichlet P\'olya inequality}
\label{d:sec:crossing-count}

\subsection{Eigenphase crossings and the Ky Fan estimate}
\label{d:sec:phase}

Proposition~\ref{d:prop:near} shows that an $m$-fold Dirichlet
resonance produces at least $m$ eigenphases of the unitary path $U_r$
arbitrarily close to $2\pi$.  We now convert this qualitative
information into a quantitative lower bound for the amount of phase
motion required to produce these crossings.

The key input is the positive radial generator
\[
L_{r,k}=-iU_r^*\partial_rU_r\ge0.
\]
Since the energy $k^2$ is fixed throughout this subsection, we write $L_r:=L_{r,k}$ throughout this subsection. For a differentiable local eigenphase branch $e^{i\theta_j(r)}$ with a corresponding normalized eigenvector $v_j(r)$, the first-variation formula gives 
\[
\theta_j'(r)
=
\langle v_j(r),L_rv_j(r)\rangle.
\]
Thus $L_r$ measures the instantaneous angular velocity of the
unitary spectrum.  If at a given radius we follow at most $m$
orthogonal phase directions, Lemma~\ref{d:lem:KyFan} controls
their total instantaneous velocity by the sum of the $m$ largest
eigenvalues of $L_r$.

Recall that, for a positive trace-class operator $A$,
$\Phi_m(A)$ denotes the sum of its $m$ largest
eigenvalues, counted with multiplicity, and
$\tau_m(A)=\tr A-\Phi_m(A)$, as defined in
\eqref{d:eq:spectral-tail}.

The following variational principle is due to
Ky Fan.
\begin{lemma}[\cite{Fan}]
    \label{d:lem:KyFan}
    Let $A$ be a positive compact self-adjoint operator on a separable
    Hilbert space $\HH$.  Then, for every $m\ge1$,
    \begin{equation}\label{d:eq:fan-variational}
        \Phi_m(A)
        =
        \sup_{\substack{P=P^*=P^2\\ \rank P=m}}
        \tr(PA).
    \end{equation}
    In particular, if $v_1,\ldots,v_q$ are orthonormal and $q\le m$,
    then
    \begin{equation}\label{d:eq:fan-upper}
        \sum_{j=1}^q\langle v_j,Av_j\rangle
        \le\Phi_m(A).
    \end{equation}
\end{lemma}

We shall use the following consequence: for positive trace-class operators $A,B$,
\begin{equation}\label{d:eq:fan-continuity}
    |\Phi_m(A)-\Phi_m(B)|
    \le
    \|A-B\|_{\mathcal S_1}
\end{equation}
which follows from
\eqref{d:eq:fan-variational} and $|\tr(P(A-B))|
\le\|A-B\|_{\mathcal S_1}$. Hence
\begin{equation}\label{d:eq:tail-continuity}
    |\tau_m(A)-\tau_m(B)|
    \le
    2\|A-B\|_{\mathcal S_1}.
\end{equation}

The Ky Fan principle has played important role in several recent works, e.g.  \cite{STW26}. Its role in the current paper is to compare the partial sum $\Phi_m(L_r)$ with the full trace of $L_r$. The latter satisfies the area identity 
\[
\int_0^1\tr L_r\,dr
=
\frac{k^2}{2}|\Omega|.
\]
For any orthonormal family of at most $m$ eigenvectors of $U_r$, the sum of their instantaneous angular velocities is bounded above by $\Phi_m(L_r)$. Integrating this bound yields the estimate for the number of eigenphase crossings. 
The difference
\[
\tau_m(L_r)=\tr L_r-\Phi_m(L_r)
\]
is strictly positive, since $L_r$ is positive, injective,
and trace class on the infinite-dimensional space $\ell^2$.
The counting argument uses only $\Phi_m(L_r)$, whereas
the area identity involves the full trace $\tr L_r$.
This positive remainder is the key to obtaining the
strict P\'olya inequality.

We now establish the estimate needed to count
the eigenphase crossings. The qualitative ingredients
are standard.  Local eigenphase branches and their first-order
variation follow from analytic perturbation theory
\cite[Chapters~II and VII]{Kato}.  Spectral flow for unitary paths
which are compact or Schatten-class perturbations of the identity is
developed, in particular, in
\cite[Section~3]{Pushnitski2001} for the unitary setting; see also \cite{Phillips1996} for general spectral-flow background in the self-adjoint Fredholm setting.

What is needed here is the following quantitative refinement:
the total multiplicity of any selected positive crossings is controlled by the integral of the Ky Fan $m$-term partial sum of the positive generator. This partial-trace estimate is the ingredient
that preserves the positive tail $\tau_m$.

We write $\mathcal B(\HH)$ for the bounded linear
operators on $\HH$. For $0<\alpha<2\pi$, by a \emph{crossing of $e^{i\alpha}$} we mean
a time $t$ such that $e^{i\alpha}\in\sigma(U(t))$.
Its multiplicity is the dimension of the corresponding eigenspace. 

\begin{lemma}
    \label{d:lem:phase}
    Let $U:[0,b]\longrightarrow\mathcal B(\HH)$ be a continuously differentiable unitary path satisfying $U(0)=I$.
    Assume that $L(t):=-iU(t)^*U'(t)$
    is positive, trace class, and trace-norm continuous, and that
    \begin{equation}\label{d:eq:strict-generator}
        \langle v,L(t)v\rangle>0
        \qquad(v\ne0,\ t>0).
    \end{equation}
    Assume moreover that $U$ is locally real analytic in operator norm
    for $t>0$.  Put $B:=\int_0^b\tr L(t)\,dt$.
    Then $U(t)-I$ is trace class for every $t$, and the following hold.

    \smallskip
    \noindent
    \textup{(i)}
    For every $0<\alpha<2\pi$, all crossings of $e^{i\alpha}$ in
    $(0,b)$ are positive and their total multiplicity $c_\alpha$ is
    finite.  For every integer $m$ with
    $1\le m\le c_\alpha$,
    \begin{equation}\label{d:eq:fan-cost}
        \alpha m
        \le
        \int_0^b\Phi_m(L(t))\,dt.
    \end{equation}
    In particular,
    \begin{equation}\label{d:eq:cost}
        \alpha c_\alpha\le B.
    \end{equation}

    \smallskip
    \noindent
    \textup{(ii)}
    For each fixed $t$, the spectrum of $U(t)$ in the open lower
    semicircle $\{e^{i\theta}:\pi<\theta<2\pi\}$ has finite total multiplicity.

    \smallskip
    \noindent
    \textup{(iii)}
    Let $0<a<b_1\le b$, $\pi<\alpha<2\pi$.
    Suppose that $U(a)$ has no spectrum with argument in
    $[\alpha,2\pi)$, that $e^{i\alpha}\notin\sigma(U(b_1))$,
    and that $U(b_1)$ has at least $m$ eigenvalues, counted with
    multiplicity, whose arguments lie in $(\alpha,2\pi)$.  Then
    $e^{i\alpha}$ is crossed at least $m$ times, counted with
    multiplicity, in $(a,b_1)$.
\end{lemma}

\begin{proof}
    We separate the quantitative Ky Fan argument from the standard
    spectral-flow facts.
   First note that $U'=iUL$ and the trace-norm continuity of $L$ give
    \[U(t)-I=i\int_0^t U(s)L(s)\,ds\in\mathcal S_1,\]
    where the integral is a Bochner integral in $\mathcal S_1$. 

    \medskip
    \noindent
    \emph{Step 1: positivity of the local crossings.}
Since $U(t)-I$ is trace class, every spectral point of $U(t)$
different from $1$ is an isolated eigenvalue of finite multiplicity.

    By analytic perturbation theory
    \cite[Chapters~II and VII]{Kato}, a spectral cluster separated from
    $1$ admits local analytic eigenphase branches.  For a simple branch
    \[
    U(t)v(t)=e^{i\theta(t)}v(t),
    \qquad
    \|v(t)\|=1,
    \]
    differentiation gives
    \begin{equation}\label{d:eq:phase-velocity}
        \theta'(t)
        =
        \langle v(t),L(t)v(t)\rangle.
    \end{equation}
    At a multiple eigenvalue, the analogue of
\eqref{d:eq:phase-velocity} states that the derivatives of the local
eigenphase branches are the eigenvalues of the compression of $L(t)$
to the corresponding eigenspace. By \eqref{d:eq:strict-generator}, this compression is positive definite.  Hence every crossing away from $1$ is locally isolated and occurs in the increasing angular direction.

    \medskip
    \noindent
    \emph{Step 2: the finite-dimensional Ky Fan estimate.}
    Consider first a finite-dimensional path satisfying the same
    differential equation.  
    Rellich--Kato perturbation theory gives local real-analytic eigenphase parameterizations, including through multiple eigenvalues after relabeling. At a multiple eigenvalue the first derivatives are the eigenvalues of the compression of $L(t)$ to the corresponding eigenspace. Under \eqref{d:eq:strict-generator} that compression is positive definite, so every local phase branch moves strictly forward. To pass from the local branches to global lifts, follow each eigenvalue continuously through collisions and, whenever a branch crosses the cut at $1$, increase its argument by $2\pi$. On every compact subinterval of $(0,b]$ the local analytic parameterizations can be patched in this way; continuity of $U(t)$ at $t=0$ then fixes the normalization. Standard absolute-continuity of eigenvalues of a $C^1$ finite-dimensional normal path gives global continuous, nondecreasing, absolutely continuous lifts
    \[
    e^{i\theta_j(t)},
    \qquad
    \theta_j(0)=0.
    \]
    At almost every $t$ the lifts are differentiable and an orthonormal eigenbasis $v_j(t)$ can be chosen compatibly; the first-variation formula then gives
    \[
    \theta_j'(t)
    =
    \langle v_j(t),L(t)v_j(t)\rangle.
    \]

    Select any $m$ crossings of $e^{i\alpha}$, counted with
    multiplicity.  A single lifted branch may contain more than one of
    the selected crossings.  Let $q_j$ denote the number of selected
    crossings on the $j$th branch, and let $J:=\{j:q_j>0\}$.
    Then
    \[
    \sum_{j\in J}q_j=m,
    \qquad
    |J|\le m.
    \]
    If a nondecreasing lifted phase contains $q_j$ crossings of
    $e^{i\alpha}$, then by time $b$ it has accumulated at least $
    \alpha+2\pi(q_j-1)
    \ge
    \alpha q_j$,
    since $\alpha<2\pi$. Therefore
    \begin{equation*}
        \alpha m
        \le
        \sum_{j\in J}\theta_j(b).
    \end{equation*}

    At almost every time the eigenvectors associated with the branches
    in $J$ form an orthonormal family of cardinality at most $m$.
    The Ky Fan bound \eqref{d:eq:fan-upper} therefore gives
    \[
    \sum_{j\in J}\theta_j'(t)
    \le
    \Phi_m(L(t)).
    \]
    After integration,
    \begin{equation}\label{d:eq:finite-fan-cost}
        \alpha m
        \le
        \int_0^b\Phi_m(L(t))\,dt.
    \end{equation}
    This is the quantitative estimate that will be used below.

    \medskip
    \noindent
    \emph{Step 3: passage to the infinite-dimensional path.}
    Choose an increasing sequence of finite-rank orthogonal projections $P_n$ with $P_n\to I$ strongly and put $L_n(t):=P_nL(t)P_n$.
    Let $U_n$ solve
    \[
    U_n'=iU_nL_n,
    \qquad
    U_n(0)=I
    \]
    on $P_n\HH$, and extend it by the identity on $P_n^\perp\HH$.
     For every $0\ne v\in P_n\HH$ and $t>0$,
     \[
        \langle v,L_n(t)v\rangle
        =\langle v,L(t)v\rangle>0,
    \]
     so each finite-dimensional approximating path satisfies the same strict positivity condition on its nontrivial finite-dimensional part.

    Since $t\mapsto L(t)$ has compact image in $\mathcal S_1$,
    strong convergence $P_n\to I$ implies
    \begin{equation}\label{d:eq:Ln-uniform}
        \sup_{0\le t\le b}
        \|L_n(t)-L(t)\|_{\mathcal S_1}
        \longrightarrow0.
    \end{equation}
    Moreover,
    \[
    \frac{d}{dt}(U_nU^*)
    =
    iU_n(L_n-L)U^*,
    \]
    so
    \begin{equation}\label{d:eq:approx}
        \sup_{0\le t\le b}
        \|U_n(t)-U(t)\|
        \le
        \int_0^b
        \|L_n(t)-L(t)\|_{\mathcal S_1}\,dt
        \longrightarrow0.
    \end{equation}

   We now verify crossing stability in the form needed here. Let $t_0$ be a crossing of $e^{i\alpha}$ of multiplicity $q$. By Step~1 and local analyticity, choose $t_-<t_0<t_+$ and $\eta>0$ so that the closed spectral arc
    \[
       \mathcal A:=\{e^{i\theta}:\alpha-\eta\le\theta\le\alpha+\eta\}
    \]
   is disjoint from $1$, its two endpoints stay in the resolvent set of $U(t)$ for $t\in[t_-,t_+]$, and the Riesz projection of $U(t)$ onto $\mathcal A$ has constant rank $q$. At $t_-$ all $q$ eigenvalues in this cluster lie on the $\theta<\alpha$ side, whereas at $t_+$ all lie on the $\theta>\alpha$ side. Uniform operator-norm convergence in \eqref{d:eq:approx} implies, for all sufficiently large $n$, equality of the corresponding Riesz-projection ranks for $U_n(t_-)$ and $U_n(t_+)$ and preserves these two side counts. Since the finite-dimensional generator $L_n$ is positive definite on $P_n\HH$, its eigenphases can cross $e^{i\alpha}$ only in the positive direction. Continuity therefore forces at least $q$ crossings of $e^{i\alpha}$ by $U_n$ on $(t_-,t_+)$. This is exactly the local spectral-projection description of unitary spectral flow; see \cite[Section~3]{Pushnitski2001}.

     For a finite collection of crossings, first group crossings occurring at the same time and then choose pairwise disjoint neighborhoods of the distinct crossing times. Applying the preceding argument simultaneously to those finitely many neighborhoods shows that, if $U$ has at least $m$ selected crossings, then for all sufficiently large $n$ the finite-dimensional path $U_n$ has at least $m$ corresponding positive crossings. Applying
    \eqref{d:eq:finite-fan-cost} and using
    \eqref{d:eq:fan-continuity} together with
    \eqref{d:eq:Ln-uniform}, we obtain
    \[
    \alpha m
    \le
    \int_0^b\Phi_m(L(t))\,dt.
    \]
    This proves \eqref{d:eq:fan-cost} for every finite collection of
    crossings.

    The same estimate also shows that the total crossing
    multiplicity is finite. Indeed, since
    \[
    \Phi_m(L(t))\le\tr L(t),
    \]
    every finite collection of $m$ crossings satisfies $\alpha m\le B$.
    Hence
    \[
    c_\alpha\le\frac{B}{\alpha}<\infty,
    \]
    and \eqref{d:eq:cost} follows.  This proves \textup{(i)}.

    \medskip
    \noindent
    \emph{Step 4: fixed-time finiteness and crossing balance.}
    For a fixed $t$, consider any finite collection of $N$ eigenvalues
    of $U(t)$ in the open lower semicircle.  By
    \eqref{d:eq:approx}, the approximating path $U_n(t)$ has, for large
    $n$, the same number of nearby eigenvalues there.  Their
    nondecreasing phase lifts have values greater than $\pi$.   If $d_n=\rank P_n$ and $\theta_{n,1},\ldots,\theta_{n,d_n}$ denote all phase lifts of the finite-dimensional path, then the selected $N$ phases form only a subcollection, while the remaining lifts are nonnegative.  Hence
    \[
    \pi N
    \le
    \sum_{j=1}^{d_n}\theta_{n,j}(t)
    =
    \int_0^t\tr L_n(s)\,ds.
    \]
    Letting $n\to\infty$ gives $\pi N\le B$.
    Since the finite collection was arbitrary, the total multiplicity
    in the lower semicircle is finite.  This proves \textup{(ii)}.

    For \textup{(iii)}, choose $\alpha<\beta<2\pi$ so that the $m$ selected eigenvalues of $U(b_1)$ lie in the open
    arc $(\alpha,\beta)$ and neither endpoint $e^{i\alpha}$ nor
    $e^{i\beta}$ belongs to the relevant endpoint spectra.  This is
    possible by part \textup{(ii)}.

    Let $n_{(\alpha,\beta)}(t)$ denote the spectral multiplicity in this arc, and let
    $c_\alpha(a,b_1)$ and $c_\beta(a,b_1)$ denote the corresponding
    crossing multiplicities.  Since all crossings are positive, a
    crossing through $\alpha$ enters the arc and a crossing through
    $\beta$ leaves it.  The usual spectral-flow balance therefore gives
    \begin{equation*}
        n_{(\alpha,\beta)}(b_1)
        -
        n_{(\alpha,\beta)}(a)
        =
        c_\alpha(a,b_1)-c_\beta(a,b_1);
    \end{equation*}
     see \cite[Section~3]{Pushnitski2001} for the unitary formulation.

    By assumption, $n_{(\alpha,\beta)}(a)=0$, $n_{(\alpha,\beta)}(b_1)\ge m$,
    and $c_\beta(a,b_1)\ge0$.
    Consequently
    \[
    c_\alpha(a,b_1)
    =
    n_{(\alpha,\beta)}(b_1)+c_\beta(a,b_1)
    \ge m.
    \]
    This proves \textup{(iii)}.
\end{proof}

\subsection{Counting resonances and endpoint multiplicities}\label{d:sec:counting}

    We now combine the preceding results in the order in which they were
    proved.  Along the conformal family
    \[
    \Omega_r=F(r\D),
    \qquad 0<r\le1,
    \]
    each ordered Dirichlet eigenvalue decreases strictly as $r$ increases by the standard strict domain-monotonicity theorem for Dirichlet eigenvalues; see \cite{Davies}. 
    Thus every eigenvalue which lies at or below the fixed energy $k^2$ at
    $r=1$ crosses that energy exactly once.  At each such resonant radius,
    Proposition~\ref{d:prop:near} produces the same number of unitary
    phases near $2\pi$. Positivity of the radial generator gives positive crossings, and Lemma~\ref{d:lem:phase} bounds the number of selected crossings in terms of the integral of the Ky Fan partial sum $\Phi_m(L_{r,k})$.

    The exact area identity involves the full trace of $L_{r,k}$.
    Consequently the difference
    \[
    \tau_m(L_{r,k})
    =
    \tr L_{r,k}-\Phi_m(L_{r,k})
    \]
    does not appear in the Ky Fan bound for the first $m$ crossings. Its strict positivity yields the strict remainder in the final P\'olya inequality.

    \begin{proposition}
        \label{d:prop:remainder}
        Let $\Omega=F(\D)$,
        where $F$ is holomorphic and injective on a neighborhood of
        $\overline\D$.  For every $k>0$ and every integer $1\le m\le N_D^{\le}(k^2)$,
        one has
        \begin{equation}\label{d:eq:count-fan}
            2\pi m
            \le
            \int_0^1\Phi_m(L_{r,k})\,dr,
        \end{equation}
        and consequently
        \begin{equation}\label{d:eq:quantitative-analytic}
            k^2|\Omega|
            \ge
            4\pi m
            +
            2\int_0^1\tau_m(L_{r,k})\,dr
            >
            4\pi m.
        \end{equation}
    \end{proposition}

    \begin{proof}
        For simplicity denote $N:=N_D^{\le}(k^2)$.
        By the standard continuity and strict domain monotonicity of Dirichlet eigenvalues, together with $\lambda_j(\Omega_r)\to+\infty$ as $r\to0+$, each 
    $1\le j\le N$ has a unique radius $\rho_j\in(0,1]$ such that
    $\lambda_j(\Omega_{\rho_j})=k^2$. 
        For $j>N$ no such crossing occurs.  Group the equal values among
        $\rho_1,\ldots,\rho_N$ as $0<r_1<\cdots<r_q\le1$,
        and let $m_\ell$ be the multiplicity of the resonance at $r_\ell$.
        Then
        \begin{equation}\label{d:eq:count}
            \sum_{\ell=1}^q m_\ell=N_D^{\le}(k^2).
        \end{equation}
        This formulation automatically includes the possibility
        $r_q=1$, so eigenvalues equal to $k^2$ at the endpoint are counted
        with their full multiplicity.

        Choose nonresonant radii
        \[
        a_1\in(0,r_1),
        \qquad
        a_\ell\in(r_{\ell-1},r_\ell),
        \quad 2\le\ell\le q.
        \]
        By Lemma~\ref{d:lem:phase}\textup{(ii)}, each operator
        $U_{a_\ell,k}$ has only finitely many eigenvalues in the open lower
        semicircle.  Since there are only finitely many radii $a_\ell$, we
        may choose a single number $\alpha_0\in(\pi,2\pi)$ so close to $2\pi$ that none of the operators $U_{a_\ell,k}$ has
        spectrum with argument in $[\alpha_0,2\pi)$.

        Fix $\alpha\in(\alpha_0,2\pi)$. For each resonance $r_\ell$, Proposition~\ref{d:prop:near} gives,
        at every sufficiently close nonresonant radius below $r_\ell$, at
        least $m_\ell$ eigenvalues of $U_{r,k}$ with arguments in
        $(\alpha,2\pi)$.  We may therefore choose
        $b_\ell\in(a_\ell,r_\ell)$
        such that $U_{b_\ell,k}$ has at least $m_\ell$ such eigenvalues.
        After an arbitrarily small adjustment of $b_\ell$, if necessary, we
        may also arrange $e^{i\alpha}\notin\sigma(U_{b_\ell,k})$,
        because crossings of $e^{i\alpha}$ are isolated.

        Lemma~\ref{d:lem:phase}\textup{(iii)} now applies on each interval
        $(a_\ell,b_\ell)$ and gives at least $m_\ell$ crossings of
$e^{i\alpha}$ there. These intervals are pairwise disjoint, so
\eqref{d:eq:count} shows that the full radial path has at least
        \[
        \sum_{\ell=1}^q m_\ell
        =
        N_D^{\le}(k^2)
        \]
        crossings of $e^{i\alpha}$.
        In particular, for every $1\le m\le N_D^{\le}(k^2)$,
        we may select $m$ of these crossings.  Lemma~\ref{d:lem:phase}
        \textup{(i)} then gives
        \[
        \alpha m
        \le
        \int_0^1\Phi_m(L_{r,k})\,dr.
        \]
        The right-hand side does not depend on $\alpha$.  Letting
        $\alpha \to2\pi-$ yields
        \[
        2\pi m
        \le
        \int_0^1\Phi_m(L_{r,k})\,dr,
        \]
        which is \eqref{d:eq:count-fan}.

        Finally, by the radial area identity
        \eqref{d:eq:radial-area},
        \[
        \frac{k^2}{2}|\Omega|
        =
        \int_0^1\tr L_{r,k}\,dr.
        \]
        Using $\tr L_{r,k} = \Phi_m(L_{r,k})+\tau_m(L_{r,k})$,
        we obtain
        \begin{align*}
            \frac{k^2}{2}|\Omega| =
            \int_0^1\Phi_m(L_{r,k})\,dr
            +
            \int_0^1\tau_m(L_{r,k})\,dr \ge
            2\pi m
            +
            \int_0^1\tau_m(L_{r,k})\,dr.
        \end{align*}
        Multiplying by $2$ proves the first inequality in
        \eqref{d:eq:quantitative-analytic}.

        The inequality is strict.  Indeed, for every $r>0$,
        Proposition~\ref{d:prop:radial} shows that $L_{r,k}$ is positive,
        trace class, and injective on the infinite-dimensional space $\HH$.
        Hence $\tau_m(L_{r,k})>0$.
        Its continuity in $r$ follows from the trace-norm continuity of
        $L_{r,k}$ and \eqref{d:eq:tail-continuity}.  Therefore
        \[
        \int_0^1\tau_m(L_{r,k})\,dr>0,
        \]
        which completes the proof.
    \end{proof}

    \subsection{Conformal exhaustion and completion of the proof}
\label{d:sec:completion} 
    Proposition~\ref{d:prop:remainder} was proved under the temporary
    assumption that the Riemann map extends holomorphically beyond
    $\overline\D$.  We now remove that assumption.

    \begin{proof}[Proof of Theorem \ref{d:thm:main}]
        Let $\Omega\subset\mathbb R^2$ be a bounded simply connected domain.
        By the Riemann mapping theorem there is a conformal bijection $
        F:\D\longrightarrow\Omega$. 
        For $0<s<1$, set $\Omega_s:=F(s\D)$.
        Then $\Omega_s\uparrow\Omega$ as $s \to1-$.
        Moreover, the map $F_s(z):=F(sz)$
        is holomorphic and injective on a neighborhood of $\overline\D$.
        Thus Proposition~\ref{d:prop:remainder} applies to every
        $\Omega_s$.
        The boundary transports obtained from $F$ on different compact radial
        intervals agree by uniqueness of the transport equations.  We may
        therefore use the notation
        \[
        L_{r,k},
        \qquad
        0<r<1,
        \]
        for the radial generator associated with the original Riemann map
        $F$.

        Fix $j\ge1$.   By the standard monotone convergence of Dirichlet forms and eigenvalues under domain exhaustion, we have the spectral convergence
        \begin{equation}\label{d:eq:eigenvalue-exhaustion}
            \lambda_j(\Omega_s)
            \longrightarrow
            \lambda_j(\Omega)
            \qquad(s \to1-).
        \end{equation}
        Set
        \begin{equation*}
            k_s:=\sqrt{\lambda_j(\Omega_s)},
            \qquad
            k:=\sqrt{\lambda_j(\Omega)}.
        \end{equation*}
        By \eqref{d:eq:eigenvalue-exhaustion},
$k_s\longrightarrow k>0$; moreover,
$|\Omega_s|\longrightarrow|\Omega|$.

        Apply Proposition~\ref{d:prop:remainder} to the conformal map
        $F_s(z)=F(sz)$, at the energy $k_s^2$, and choose $m=j$.
        This is legitimate because $
        N_{D,\Omega_s}^{\le}(k_s^2)\ge j$,
        including the full multiplicity of the eigenvalue
        $\lambda_j(\Omega_s)$.

        We now relate the radial generator for $F_s$ to that for $F$.
        If $t\in[0,1]$ is the radial parameter for $F_s$, then
        \[
        F_s(t\D)=F(st\D)=\Omega_{st}.
        \]
        Hence its unitary path is $t\longmapsto U_{st,k_s}$.
        Differentiating with respect to $t$ shows that its generator is $\widetilde L^{(s)}_{t,k_s}=sL_{st,k_s}$.
        The same reparametrization applies to the base-point correction
        $R_r$, since it is defined by its radial differential equation.

        Because $\tau_j$ is positively homogeneous,
        \[
        \tau_j(sA)=s\,\tau_j(A)
        \qquad(s\ge0),
        \]
        Proposition~\ref{d:prop:remainder} therefore gives
        \begin{equation}\label{d:eq:exhaustion-remainder}
            |\Omega_s|\lambda_j(\Omega_s)
            \ge
            4\pi j
            +
            2\int_0^s\tau_j(L_{r,k_s})\,dr.
        \end{equation}

        We now pass to the limit while retaining this remainder.
        For each fixed $0<r<1$,  choose $b$ with $r<b<1$ and apply Proposition~\ref{d:prop:radial} to the map $F_b(z)=F(bz)$.  Since the transport constructed from $F_b$ is just the reparametrized restriction of that constructed from $F$, the proposition gives the required local joint trace-norm continuity of $L_{r,k}$ on compact subsets of $(0,1)\times(0,\infty)$.  Together with \eqref{d:eq:tail-continuity},  this gives
        \begin{equation}\label{d:eq:tail-pointwise}
            \tau_j(L_{r,k_s})
            \longrightarrow
            \tau_j(L_{r,k})
            \qquad(s \to1-).
        \end{equation}

        Choose any sequence $s_n \to1-$.
        On $(0,1)$ define the nonnegative functions
        \[
        f_n(r)
        :=
        \mathbf 1_{(0,s_n)}(r)\,
        \tau_j(L_{r,k_{s_n}}).
        \]
        For every fixed $r\in(0,1)$, we have eventually $r<s_n$, and
        \eqref{d:eq:tail-pointwise} gives
        \[
        f_n(r)\longrightarrow\tau_j(L_{r,k}).
        \]
        Fatou's lemma and \eqref{d:eq:exhaustion-remainder} therefore imply
        \begin{equation}\label{d:eq:quantitative-rough}
            |\Omega|\lambda_j(\Omega)
            \ge
            4\pi j
            +
            2\int_0^1\tau_j(L_{r,k})\,dr.
        \end{equation}

        The integral is finite.  Indeed,
        \[
        0\le\tau_j(L_{r,k})\le\tr L_{r,k},
        \]
        and the same localization to $F_b$ shows that  the radial area identity on every compact interval
        $[0,b]\subset[0,1)$ gives
        \[
        \int_0^b\tr L_{r,k}\,dr
        =
        \frac{k^2}{2}|F(b\D)|.
        \]
        Letting $b \to1-$ and using monotone convergence yields
        \[
        \int_0^1\tr L_{r,k}\,dr
        =
        \frac{k^2}{2}|\Omega|.
        \]

        More importantly, the remainder in
        \eqref{d:eq:quantitative-rough} is strictly positive.  For every
        $r\in(0,1)$, Proposition~\ref{d:prop:radial} gives that
        $L_{r,k}$ is positive and injective on the infinite-dimensional
        space $\HH$. Hence \eqref{d:eq:positive-tail} gives
$\tau_j(L_{r,k})>0$.
        Choose any $r_0\in(0,1)$.  By trace-norm continuity,
        $r\mapsto\tau_j(L_{r,k})$ is continuous near $r_0$, and therefore
        remains positive on a nontrivial interval around $r_0$.  Consequently,
        \[
        \int_0^1\tau_j(L_{r,k})\,dr>0.
        \]
        Equation \eqref{d:eq:quantitative-rough} now gives $
        |\Omega|\lambda_j(\Omega)>4\pi j$,
        which is \eqref{d:eq:strict-eigenvalues}.

        Finally let $E>0$.  If $N_D^{\le}(E)=0$,
        then $N_D^{\le}(E)<\frac{|\Omega|E}{4\pi}$
        is immediate.  Otherwise set $ M:=N_D^{\le}(E)$.
        Then $\lambda_M(\Omega)\le E$,
        and the strict eigenvalue inequality gives
        \[
        4\pi M
    <
        |\Omega|\lambda_M(\Omega)
        \le
        |\Omega|E.
        \]
        Hence $N_D^{\le}(E)< \frac{|\Omega|E}{4\pi}$,
        which is \eqref{d:eq:polya}.  This completes the Dirichlet proof.
    \end{proof}

    \begin{remark}
        The remainder in \eqref{d:eq:quantitative-rough} depends on the
        domain, the chosen conformal parametrization, and the index $j$.
        Only its strict positivity is used here; no lower bound depending
        only on $j$ is asserted.
    \end{remark}

\end{document}